\documentclass[11pt]{amsart}

\usepackage[utf8]{inputenc}
\usepackage[T1]{fontenc}

\usepackage{amsmath,amssymb,amsfonts,textcomp,amsthm,xifthen,graphicx,color,pgfplots}

\usepackage{enumerate}
\usepackage{enumitem}

\usepackage{fullpage}

\usepackage[utf8]{inputenc}

\usepackage[pdftex,
            pdfauthor={Thomas F\"uhrer, Norbert Heuer and Michael Karkulik},
            pdftitle={MINRES for second-order PDEs with singular data},
            ]{hyperref}

\usepackage{pgfplots}
\usepgfplotslibrary{colorbrewer}

\newtheorem{theorem}{Theorem}
\newtheorem{lemma}[theorem]{Lemma}
\newtheorem{corollary}[theorem]{Corollary}
\newtheorem{proposition}[theorem]{Proposition}
\newtheorem{remark}[theorem]{Remark}

\newcommand{\patch}{\omega}
\newcommand{\Patch}{\Omega}

\newcommand\sign{\operatorname{sign}}

\newcommand\osc{\mathrm{osc}}

\def\eps{\varepsilon}

\def\di{\mathrm{d}}

\DeclareMathOperator{\supp}{supp}

\DeclareMathOperator*{\argmin}{arg\, min}

\newcommand{\ip}[2]{(#1\hspace*{.5mm},#2)}
\newcommand{\dual}[2]{\langle#1\hspace*{.5mm},#2\rangle}

\newcommand{\norm}[3][]{#1\|#2#1\|_{#3}}

\newcommand{\snorm}[2]{|#1|_{#2}}

\newcommand{\diam}{\mathrm{diam}}

\def\div{{\rm div\,}}

\newcommand{\Hdivset}[1]{\boldsymbol{H}(\div;#1)}

\newcommand{\HdivsetN}[1]{\boldsymbol{H}_N(\div;#1)}

\newcommand{\aalpha}{\boldsymbol{\alpha}}

\newcommand{\set}[2]{\big\{#1\,:\,#2\big\}}

\newcommand{\pwnabla}{\nabla_\TT}
\newcommand{\pwdiv}{\div_\TT}

\newcommand{\RT}{\ensuremath{\mathcal{RT}}}

\newcommand{\R}{\ensuremath{\mathbb{R}}}
\newcommand{\N}{\ensuremath{\mathbb{N}}}

\newcommand{\vv}{\ensuremath{\boldsymbol{v}}}
\newcommand{\ww}{\ensuremath{\boldsymbol{w}}}
\newcommand{\TT}{\ensuremath{\mathcal{T}}}

\newcommand{\cS}{\ensuremath{\mathcal{S}}}

\newcommand{\PP}{\ensuremath{\mathcal{P}}}

\newcommand{\OO}{\ensuremath{\mathcal{O}}}
\newcommand{\EE}{\ensuremath{\mathcal{E}}}

\newcommand{\normal}{\ensuremath{{\boldsymbol{n}}}}

\newcommand{\VV}{\ensuremath{\mathcal{V}}}

\newcommand\tracediv[1]{\mathrm{tr}_{#1}^{\mathrm{div}}}
\newcommand\tracegrad[1]{\mathrm{tr}_{#1}^{\mathrm{grad}}}

\newcommand{\ssigma}{{\boldsymbol\sigma}}
\newcommand{\ttau}{{\boldsymbol\tau}}

\newcommand{\cchi}{{\boldsymbol\chi}}

\newcommand{\uu}{\boldsymbol{u}}

\begin{document}

\title{MINRES for second-order PDEs with singular data}
\date{\today}

\author{Thomas F\"uhrer}

\author{Norbert Heuer}
\address{Facultad de Matem\'{a}ticas, Pontificia Universidad Cat\'{o}lica de Chile, Santiago, Chile}
\email{\{tofuhrer,nheuer\}@mat.uc.cl}

\author{Michael Karkulik}
\address{Departamento de Matem\'atica, Universidad T\'ecnica Federico Santa Mar\'ia, Valpara\'iso, Chile}
\email{michael.karkulik@usm.cl}

\thanks{{\bf Acknowledgment.} 
This work was supported by ANID through FONDECYT projects and 1210391 (TF), 1190009 (NH) and 1210579 (MK)}

\keywords{Minimum residual method, least-squares method, discontinuous Petrov--Galerkin method, singular data}
\subjclass[2010]{65N30, 
                 65N12 
                 }
\begin{abstract}
  Minimum residual methods such as the least-squares finite element method (FEM) or the discontinuous Petrov--Galerkin method with optimal test functions (DPG) usually exclude singular data, e.g., non square-integrable loads.  We consider a DPG method and a least-squares FEM for the Poisson problem.
  For both methods we analyze regularization approaches that allow the use of $H^{-1}$ loads, and also study the case of point loads. 
  For all cases we prove appropriate convergence orders. 
  We present various numerical experiments that confirm our theoretical results. 
  Our approach extends to general well-posed second-order problems.
\end{abstract}
\maketitle

\section{Introduction}
The motivation of this work is to analyze minimum residual finite element methods (MINRES FEM) with source functionals in $H^{-1}(\Omega)$, the dual space of $H_0^1(\Omega)$ with $\Omega\subset \R^d$ ($d=2,3$) a polytopal domain, and point sources.
Many of the popular MINRES FEM suffer from the fact that minimization is considered with respect to a stronger norm than the natural norm induced by the underlying PDE and, thus, is often not suited/defined for the use of singular load terms. 
Throughout, we focus on the Poisson problem but stress that our proposed methods extend to general second-order elliptic scalar problems, linear elasticity, or Stokes-type problems, provided that regularity results are available (see Appendix~\ref{sec:extension}).
Now, to illustrate the complications when applying MINRES to a problem with singular data, let us consider the following least-squares finite element method of a first-order reformulation (FOSLS) of the Poisson problem,
\begin{align}\label{intro:lsfem}
  (u_h,\ssigma_h) = \argmin_{(v_h,\ttau_h)\in W_h} \norm{\div\ttau_h + f}{}^2 + \norm{\nabla v_h-\ttau_h}{}^2.
\end{align}
Here, $W_h\subseteq H_0^1(\Omega)\times \Hdivset\Omega$ is a lowest-order discretization space and $\norm\cdot{}$ denotes the $L^2(\Omega)$ norm (details on the definition of spaces and norms can be found below).
Well-posedness of the FOSLS~\eqref{intro:lsfem} has been analyzed in~\cite{CaiLazarovManteuffelMcCormickPart1}.
A comprehensive overview and systematic approach to least-squares FEM is given in~\cite{BochevGunzberger09}.

If $f\in L^2(\Omega)$, then minimization problem~\eqref{intro:lsfem} is well defined and, provided that $f$ is sufficiently regular and $\Omega$ is convex, the solution converges at the optimal rate, i.e., 
\begin{align}\label{intro:lsfem:apriori}
  \norm{u-u_h}{H^1(\Omega)} + \norm{\ssigma-\ssigma_h}{\Hdivset\Omega} = \OO(h),
\end{align}
where $u\in H_0^1(\Omega)$ solves $\Delta u = -f$ and $\ssigma =\nabla u$.

If $f\in H^{-1}(\Omega)$, then minimization problem~\eqref{intro:lsfem} does not make sense. 
In this article, we propose to replace $f\in H^{-1}(\Omega)$ by a regularized functional $Q_h^\star f\in L^2(\Omega)$, where $Q_h^\star$ is a computable quasi-interpolation operator. Instead of~\eqref{intro:lsfem} we consider the regularized problem
\begin{align}\label{intro:lsfem:reg}
  (u_h,\ssigma_h) = \argmin_{(v_h,\ttau_h)\in W_h} \norm{\div\ttau_h + Q_h^\star f}{}^2 + \norm{\nabla v_h-\ttau_h}{}^2
\end{align}
and show that its solution converges optimally in the sense that 
\begin{align}\label{intro:lsfem:reg:apriori}
  \norm{u-u_h}{H^1(\Omega)} + \norm{\ssigma-\ssigma_h}{} \lesssim h^s \norm{f}{H^{-1+s}(\Omega)}
\end{align}
for $s\in[0,1]$ depending on $\Omega$ and $f$.

An alternative approach is to consider weaker norms from the beginning, as is done in the seminal work~\cite{BLP97}, i.e., one aims at solving 
\begin{align*}
  (u_h,\ssigma_h) = \argmin_{(v_h,\ttau_h)\in W_h} \norm{\div\ttau_h+f_h}{H^{-1}(\Omega),h}^2 + \norm{\nabla v_h-\ttau_h}{}^2.
\end{align*}
Here, $\norm\cdot{H^{-1}(\Omega),h}$ denotes a discrete $H^{-1}(\Omega)$ norm and $f_h$ is some suitable discretized load. 
One advantage of our proposed regularization approach is that the same convergence rates under the same regularity assumptions as in~\cite[Corollary~3.1]{BLP97} are achieved, though at a lower cost. 

In~\cite{2021MillarMugaRojasVdZee} a minimum residual method in Banach spaces to obtain a projection of the functional $f$ in a polynomial space is proposed, 
whereas the construction of our regularization operators $Q_h^\star$ is based on the adjoint of an operator that also appears in the related work~\cite{ErnZanotti20}. There, the authors consider a different philosophy by smoothing test functions instead of regularizing the load. 
Here, we consider two choices for $Q_h^\star$, both are bounded as operators $H^{-1}(\Omega)\to H^{-1}(\Omega)$ and $L^2(\Omega)\to L^2(\Omega)$, idempotent on piecewise constant functions, have approximation properties and are computable with linear cost depending on the number of elements of the mesh.

As mentioned before, we also study the discontinuous Petrov--Galerkin method with optimal test functions (DPG). It has been introduced by Demkowicz \& Gopalakrishnan in~\cite{partI,partII,partIII}. Particularly, a DPG method for an ultraweak reformulation of the Poisson problem is studied in~\cite{DPGpoisson}. 
DPG methods are MINRES methods that minimize a functional in the dual norm of broken test spaces. 
In general, they suffer from the same difficulties as described above when trying to use $H^{-1}$ or even less regular loads. 
We extend and analyze the DPG method for such data. 

\subsection{Novel contributions}
We show that the two aforementioned MINRES FEM for the Poisson problem on Lipschitz domains can be modified to handle $H^{-1}$ loads and lead to optimal convergence rates, see Theorem~\ref{thm:dpgPoisson} (DPG) and Theorem~\ref{thm:LSregularized} (FOSLS), respectively. 
For the DPG method we consider a local postprocessing to obtain even higher rates for the primal variable (Theorem~\ref{thm:dpgPostProcess}).

Moreover, we show that the built-in error estimators of the MINRES FEM are --- up to oscillation terms --- equivalent to the error and, thus, can be used to steer an adaptive algorithm, see Theorem~\ref{thm:dpgPoisson:aposteriori} (DPG) and Theorem~\ref{thm:lsq:aposteriori} (FOSLS).

The theory developed for our proposed regularization approach also extends to point sources, see Theorem~\ref{thm:dpg:delta} (DPG) resp. Theorem~\ref{thm:lsq:delta} (FOSLS).
For the convergence analysis of classical finite element methods with point loads we refer to the seminal work~\cite{Scott75}.

We also report on novel results concerning optimal $L^2(\Omega)$ rates for the error in the primal variable in the FOSLS with lowest-order approximation spaces. 
To be more precise, we show under a condition on the mesh that the regularized FOSLS approach~\eqref{intro:lsfem:reg} implies optimal rates, see Corollary~\ref{cor:ls:L2}. 
Contrary, the standard FOSLS approach~\eqref{intro:lsfem} does not produce optimal rates which we verify by a numerical example, see Section~\ref{sec:ex:L2notoptimal}. 
Prior works on optimal $L^2(\Omega)$ error rates include~\cite{LSQL2normSmooth,Ku2011}.

\subsection{Outlook}
The remainder of this work is structured as follows: 
In Section~\ref{sec:minres} we introduce notation as well as two MINRES methods for the Poisson problem. 
We recall results from the literature on a DPG method (Section~\ref{sec:poisson}) and a FOSLS method for the Poisson problem (Section~\ref{sec:LSFEM}).
In Section~\ref{sec:regularization} we introduce the regularization operators $Q_h^\star$, propose and analyze regularized variants of the two aforementioned MINRES methods. 
Section~\ref{sec:L2errLS} deals with $L^2(\Omega)$ errors of the primal variable in the FOSLS method.
Convergence of the DPG and FOSLS methods for point loads is analyzed in Section~\ref{sec:pointloads}.
The final Section~\ref{sec:ex} contains numerical experiments and Appendix~\ref{sec:extension} shows how to extend the techniques to general second-order problems.

\section{MINRES FEM}\label{sec:minres}

\subsection{Sobolev spaces and broken variants}
For a bounded Lipschitz domain $\omega\subset \R^d$ ($d=2,3$)
let $H_0^n(\omega)$ and $H^n(\omega)$ denote the usual Sobolev spaces for $n\in\N$ equipped with the norm and seminorm
\begin{align*}
  \norm{u}{H^n(\omega)}^2 &= \norm{u}{\omega}^2 + \norm{D^n u}{\omega}^2, \quad 
  \snorm{u}{H^n(\omega)} = \norm{D^n u}{\omega} \text{ for } u\in H^n(\omega).
\end{align*}
Here, $D^n$ stands for the $n$-th order weak derivatives and $\norm{\cdot}{\omega}$ denotes the $L^2(\omega)$ norm with inner product $\ip{\cdot}{\cdot}_\omega$. 
If $\omega=\Omega$ we skip the index in the notation of the $L^2(\Omega)$ norm and $L^2(\Omega)$ inner product. 
Note that by Poincar\'e inequalities we have that $\snorm{u}{H^n(\omega)} \eqsim \norm{u}{H^n(\omega)}$ for $u\in H_0^n(\omega)$. 

We consider the intermediate Sobolev spaces $H^s(\omega)$ and $H_0^s(\omega)$ for noninteger $s$ defined by interpolation.
The notation for the dual spaces is $H^{-s}(\omega) = (H_0^s(\omega))'$, $\widetilde H^{-s}(\omega) = (H^s(\omega))'$ with $L^2(\omega)$ as pivot space. 
We are particularly interested in $H_0^1(\Omega)$ and its dual $H^{-1}(\Omega)$ where we define the dual norm by using the norm $\norm{\nabla\cdot}{}$ on $H_0^1(\Omega)$, i.e.,
\begin{align*}
  \norm{\phi}{H^{-1}(\Omega)} = \sup_{0\neq v\in H_0^1(\Omega)}\frac{\ip{\phi}{v}}{\norm{\nabla v}{}}.
\end{align*}
Here, $\ip\cdot\cdot$ denotes the $H^{-1}(\Omega)\times H_0^1(\Omega)$ duality bracket which for regular enough arguments reduces to the $L^2(\Omega)$ scalar product.

Furthermore, 
\begin{align*}
  \Hdivset\omega &:= \set{\ssigma\in L^2(\omega)^d}{\div\ssigma\in L^2(\omega)}
\end{align*}
with norm $\norm{\ttau}{\Hdivset\omega}^2:= \norm{\ttau}{\omega}^2 + \norm{\div\ttau}\omega^2$. 

For the DPG method below we use broken variants of these spaces, e.g., 
\begin{align*}
  H^1(\TT) := \prod_{T\in\TT} H^1(T) \eqsim \set{v\in L^2(\Omega)}{v|_T \in H^1(T)\,\forall T\in\TT}
\end{align*}
with norm $\norm{v}{H^1(\TT)}^2 = \sum_{T\in\TT} \norm{v}{H^1(T)}^2$. 
To simplify notation we use piecewise differential operators $\nabla_\TT\colon H^1(\TT)\to L^2(\Omega)$, $\nabla_\TT v|_T = \nabla(v|_T)$ for $T\in\TT$. 
Then, 
\begin{align*}
  \ip{\nabla_\TT v}{\ttau} = \sum_{T\in\TT} \ip{\nabla (v|_T)}{\ttau}_T \quad\forall v\in H^1(\TT), \,\ttau\in L^2(\Omega)^d. 
\end{align*}
In the same spirit we define $\Hdivset\TT$, $\div_\TT$, and $\norm{\ttau}{\Hdivset\TT}$.

\subsection{Approximation spaces}\label{sec:approx}
Let $\TT$ denote a regular (in the sense of Ciarlet) simplicial mesh of the bounded polyhedral Lipschitz domain $\Omega\subset \R^d$ ($d=2,3$) with mesh-size function $h_\TT \in L^\infty(\Omega)$, $h_\TT|_T = \diam(T)$ for all $T\in\TT$ and $h:=\max_{T\in\TT} h_\TT|_T$. The collection of $\partial T$ of all elements $T\in\TT$ is called skeleton $\cS$.
We use $\EE_T$ for the sides of an element $T$.
Vertices of the mesh are denoted by $\VV$, vertices of an element $T\in\TT$ by $\VV_T$ and interior vertices by $\VV_0 := \VV\cap\Omega$.
The patch $\patch(S)\subset\TT$ for any $S\subset\overline\Omega$ is the collection of all elements $T\in\TT$ with $S\cap\overline T\neq 0$. 
If $S$ is a singleton $S=\{s\}$, $s\in\overline\Omega$, then we simply use the notation, $\patch(s):= \patch(S)$. 
The domain associated with $\patch(S)$ is denoted by $\Patch(S)$.

Piecewise polynomial spaces of degree $\leq p\in\N_0$ are denoted by $\PP^p(\TT)$ and $\Pi_h^p\colon L^2(\Omega)\to\PP^p(\TT)$ is the corresponding $L^2(\Omega)$ orthogonal projection.

The lowest-order Raviart--Thomas space is denoted by $\RT^0(\TT)$ and we make use of the local quasi-interpolation operator constructed in~\cite[Section~3.1]{egsv2019}, denoted by $\Pi_h^\mathrm{div}$. 
It is a projection and has the following two properties, see~\cite[Theorem~3.2]{egsv2019},
\begin{align}\label{eq:propertiesPihDiv}
  \Pi_h^0\div\ttau = \div\Pi_h^\mathrm{div}\ttau, \qquad
  \norm{\Pi_h^\mathrm{div}\ttau}{} \lesssim \norm{\ttau}{} + \norm{h_\TT(1-\Pi_h^0)\div\ttau}{}
\end{align}
for all $\ttau\in \Hdivset\Omega$. Moreover, one concludes from~\cite[Theorem~3.6]{egsv2019} that, for $s\in[0,1]$,
\begin{align*}
  \norm{(1-\Pi_h^\mathrm{div})\ttau}{} \lesssim h^s\norm{\ttau}{H^s(\Omega)} + \norm{h_\TT\div\ttau}{}.
\end{align*}

\subsection{Notation}
We write $A\lesssim B$ if there exists a constant $C>0$ with $A\leq C\cdot B$ and $C$ is independent of quantities of interest (mesh-size $h$, norms of functions). 
In most of the estimates below $C$ depends on $\Omega$ and the shape-regularity of $\TT$.
We write $A\eqsim B$ if $A\lesssim B$ and $B\lesssim A$. 
To simplify the presentation in some proofs on a priori convergence rates we assume that $\TT$ is a quasi-uniform mesh, i.e., $h/\min_{T\in\TT}h_T \eqsim 1$.
We stress that for the analysis of a posteriori error estimators (Sections~\ref{sec:dpg:aposteriori} and~\ref{sec:lsq:aposteriori}) this assumption is not needed.

\subsection{DPG for Poisson}\label{sec:poisson}
We work with the first-order reformulation of the Poisson problem with homogeneous Dirichlet boundary condition on $\Gamma=\partial\Omega$, 
\begin{subequations}\label{poisson}
\begin{align}
  -\div\ssigma &= f, \\
  \ssigma - \nabla u &= 0, \\
  u|_\Gamma &= 0.
\end{align}
\end{subequations}
For solutions $u\in H_0^1(\Omega)$ of the Poisson problem there exists $1/2<s_\Omega\leq 1$ only depending on $\Omega$ with
\begin{align}\label{poisson:regshift}
  \norm{u}{H^{1+s}(\Omega)} \lesssim \norm{f}{H^{-1+s}(\Omega)} \quad\text{for } s\in[0,s_\Omega],
\end{align}
see, e.g.,~\cite{grisvard}. In particular, if $\Omega$ is convex, then $s_\Omega = 1$.

The ultraweak formulation is derived by testing with local test functions and integrating by parts. This requires to introduce trace variables that live on the skeleton $\cS$.
To that end define the trace operators
\begin{align*}
  \tracegrad{}\colon H^1(\Omega)\to (\Hdivset\TT)' \quad\text{and}\quad
  \tracediv{}\colon \Hdivset\Omega\to (H^1(\TT))'
\end{align*}
by 
\begin{align*}
  \dual{\tracegrad{}u}{\ttau}_{\cS} &= \ip{u}{\div_\TT\ttau} + \ip{\nabla u}{\ttau}, \\
  \dual{\tracediv{}\ssigma}{v}_{\cS} &= \ip{\ssigma}{\nabla_\TT v} + \ip{\div\ssigma}{v}
\end{align*}
and set 
\begin{align*}
 H_{00}^{1/2}(\cS) := \tracegrad{}(H_0^1(\Omega)), \quad H^{-1/2}(\cS) := \tracediv{}(\Hdivset\Omega).
\end{align*}
These spaces are closed with respect to the canonical norms (``minimum energy extension norms'')
\begin{align*}
  \norm{\widehat u}{1/2,\cS} &:= \inf\set{\norm{u}{H^1(\Omega)}}{u\in H^1(\Omega),\, \tracegrad{}u = \widehat u}, \\
  \norm{\widehat\sigma}{-1/2,\cS} &:= \inf\set{\norm{\ssigma}{\Hdivset\Omega}}{\ssigma\in\Hdivset\Omega, \,\tracediv{}\ssigma = \widehat \sigma},
\end{align*}
and we refer the reader to~\cite{breakSpace} for details. 
By definition we have that
\begin{align*}
  \norm{\widehat u}{1/2,\cS} \leq \norm{u}{H^1(\Omega)} \quad\forall u\in H^1(\Omega) \text{ with } \tracegrad{} u = \widehat u
\end{align*}
and the analogous estimate for the other trace space. 

We introduce the spaces
\begin{align*}
  U &:= L^2(\Omega)\times L^2(\Omega)^d \times H_{00}^{1/2}(\cS) \times H^{-1/2}(\cS), \\
  V &:= H^1(\TT) \times \Hdivset\TT,
\end{align*}
with norms
\begin{align*}
  \norm{\uu}U^2 &:= \norm{u}{}^2 + \norm{\ssigma}{}^2 + \norm{\widehat u}{1/2,\cS}^2 + \norm{\widehat \sigma}{-1/2,\cS}^2, \\
  \norm{\vv}V^2 &:= \norm{v}{H^1(\TT)}^2 + \norm{\ttau}{\Hdivset\TT}^2
\end{align*}
for $\uu= (u,\ssigma,\widehat u,\widehat \sigma)\in U$, $\vv = (v,\ttau)\in V$,
the bilinear form $b\colon U\times V \to\R$ 
\begin{align*}
  b(\uu,\vv) &:= \ip{u}{\pwdiv \ttau} + \ip{\ssigma}{\pwnabla v+\ttau} 
  -\dual{\widehat u}{\ttau}_\cS- \dual{\widehat\sigma}v_{\cS}, 
\end{align*}
and the load functional $F\colon V\to\R$, 
\begin{align*}
  F(\vv) := \ip{f}v.
\end{align*}
The inner product on $V$ reads
\begin{align*}
  \ip{\vv}{\ww}_{V} := \ip{\pwnabla v}{\pwnabla w} + \ip{v}w
    + \ip{\pwdiv\ttau}{\pwdiv\cchi} + \ip{\ttau}{\cchi}
\end{align*}
where $\vv=(v,\ttau),\ww = (w,\cchi)\in V$.

Then, the ultraweak formulation of~\eqref{poisson} is
\begin{align}\label{dpg:poisson:cont}
  u\in U\colon \quad b(\uu,\vv) = F(\vv) \quad\forall \vv\in V,
\end{align}
and admits a unique solution, cf.~\cite[Section~4]{DPGpoisson}.

We consider the lowest-order approximation spaces 
\begin{align*}
  U_h &:= \PP^0(\TT)\times [\PP^0(\TT)]^d \times \tracegrad{}(\PP^1(\TT)\cap H_0^1(\Omega)) \times  \tracediv{}(\RT^0(\TT)),  \\
  V_h &:= \PP^d(\TT)\times [\PP^2(\TT)]^d.
\end{align*}

The DPG method then reads: Given $f\in L^2(\Omega)$, find $\uu_h=(u_h,\ssigma_h,\widehat u_h,\widehat \sigma_h)\in U_h$ such that
\begin{align}\label{dpg:poisson}
  b(\uu_h,\Theta_h\ww_h) = F(\Theta_h\ww_h)\quad\forall \ww_h\in U_h,
\end{align}
where $\Theta_h\colon U_h\to V_h$ is the (discrete) trial-to-test operator, 
\begin{align*}
  \ip{\Theta_h\ww_h}{\vv_h}_V = b(\ww_h,\vv_h) \quad\forall \vv_h\in V_h.
\end{align*}
It was shown in~\cite[Theorem~2.1 and Section~3]{practicalDPG} that problem~\eqref{dpg:poisson} admits a unique solution which is quasi-optimal:
\begin{proposition}[{\cite[Theorem~3.4]{practicalDPG}}]
  Let $\uu\in U$ and $\uu_h\in U_h$ denote the solutions of~\eqref{dpg:poisson:cont} and~\eqref{dpg:poisson}, respectively. We have that
  \begin{align*}
    \norm{\uu-\uu_h}U \lesssim \min_{\ww_h\in U_h} \norm{\uu-\ww_h}U.
  \end{align*}
\end{proposition}
A priori convergence estimates are based upon the latter result, e.g.,~\cite[Corollary~3.5]{practicalDPG} proves that
\begin{align*}
  \norm{\uu-\uu_h}U \lesssim h\big(|u|_{H^2(\Omega)} + |\ssigma|_{H^2(\Omega)}\big).
\end{align*}
This estimate is often too pessimistic as it requires that $\ssigma=\nabla u \in H^2(\Omega)$ which is in general not true for solutions of the Poisson problem with right-hand side $f\in L^2(\Omega)$.
The reason is that the term involving the $H^{-1/2}(\cS)$ norm in $\norm\cdot{U}$ is estimated with a stronger norm than necessary. 
In~\cite{SupConv1} and~\cite{SupConv2} this estimate has been improved:
\begin{proposition}[{\cite[Corollary~6]{SupConv1} and \cite[Theorem~6]{SupConv2}}]\label{apriori:poisson}
  Let $f\in L^2(\Omega)$ and $\uu \in U$, $\uu_h\in U_h$ denote the solution of~\eqref{dpg:poisson:cont} and~\eqref{dpg:poisson}, respectively. The estimate 
  \begin{align*}
    \norm{\uu-\uu_h}U \lesssim h^{s_\Omega}\norm{f}{H^{-1+s_\Omega}(\Omega)} + h\norm{f}{}
  \end{align*}
  holds true. Here, $1/2<s_\Omega\leq 1$ denotes the regularity shift of the Poisson problem~\eqref{poisson:regshift}.
\end{proposition}

\subsection{FOSLS for Poisson}\label{sec:LSFEM}
Recall the first-order system~\eqref{poisson}. 
The lowest-order least-squares FEM seeks $\uu_h = (u_h,\ssigma_h)\in W_h := \PP^1(\TT)\cap H_0^1(\Omega) \times \RT^0(\TT)\subset W:= H_0^1(\Omega)\times \Hdivset\Omega$ such that
\begin{align}\label{lsq:minimization:discrete}
  \uu_h = \argmin_{\vv_h=(v_h,\ttau_h)\in W_h}\big( \norm{\nabla v_h-\ttau_h}{}^2 + \norm{\div\ttau_h+f}{}^2 \big).
\end{align}
It is well known that this minimization problem admits a unique solution. We refer the interested reader to~\cite{BochevGunzberger09} and references therein. 

We equip the space $W$ with the product space norm
\begin{align*}
  \norm{(u,\ssigma)}W^2:= \norm{u}{H^1(\Omega)}^2 + \norm{\ssigma}{\Hdivset\Omega}^2
\end{align*}
and note that $\norm{(u,\ssigma)}W^2\eqsim \norm{\nabla u-\ssigma}{}^2 + \norm{\div\ssigma}{}^2$, see, e.g.,~\cite[Theorem~3.1]{CaiLazarovManteuffelMcCormickPart1}.

The following result follows standard arguments (quasi-best approximation and approximation properties).
Details can be found, e.g., in~\cite[Theorem~5.30 and Corollary~5.31]{BochevGunzberger09}.
\begin{proposition}\label{apriori:lsq1}
  For given $f\in L^2(\Omega)$, let $\uu=(u,\ssigma)\in W$ denote the solution to the Poisson problem~\eqref{poisson} and $\uu_h\in W_h$ the solution of~\eqref{lsq:minimization:discrete}. 
  With $1/2<s_\Omega\leq 1$ from~\eqref{poisson:regshift} we have that
  \begin{align*}
    \norm{\uu-\uu_h}{W} \lesssim h^{s_\Omega}\norm{u}{H^{1+s_\Omega}(\Omega)} + \norm{(1-\Pi_h^0)f}{}.
  \end{align*}
\end{proposition}
This result means that one requires higher regularity of $f$ to conclude convergence rates. 
This is a drawback of least-squares FEM since in general, e.g., on convex domains where $s_\Omega=1$,
a priorily we can only ensure  $\norm{(1-\Pi_h^0)f}{} =\OO(h)$ if $f\in H^1(\TT)$.
However, if we do not consider the error term $\norm{\div(\ssigma-\ssigma_h)}{}$ in the estimate, we get better bounds:
\begin{proposition}\label{apriori:lsq2}
  Under the assumptions of Proposition~\ref{apriori:lsq1} the estimates
  \begin{align*}
    \norm{u-u_h}{H^1(\Omega)} + \norm{\ssigma-\ssigma_h}{} \lesssim h^{s_\Omega} \norm{\Pi_h^0f}{H^{-1+s_\Omega}(\Omega)} + h\norm{(1-\Pi_h^0)f}{} \lesssim h^{s_\Omega} \norm{f}{} + h\norm{(1-\Pi_h^0)f}{}
  \end{align*}
  hold true.
\end{proposition}
\begin{proof}
  Consider $\widetilde u$ to be the solution of
  \begin{align*}
    -\Delta \widetilde u = \Pi_h^0 f, \quad \widetilde u|_\Gamma = 0
  \end{align*}
  and $\widetilde \uu = (\widetilde u,\nabla\widetilde u)\in W$. 
  Then, by the triangle inequality
  \begin{align}\label{apriori:lsq1:proof1}
    \norm{u-u_h}{H^1(\Omega)} + \norm{\nabla u-\ssigma_h}{} \leq \norm{u-\widetilde u}{H^1(\Omega)} + \norm{\nabla u-\nabla \widetilde u}{}
    + \norm{\widetilde u-u_h}{H^1(\Omega)} + \norm{\nabla \widetilde u-\ssigma_h}{}.
  \end{align}
  The first two terms on the right-hand side are estimated by
  \begin{align*}
    \norm{u-\widetilde u}{H^1(\Omega)} + \norm{\nabla u-\nabla \widetilde u}{} \lesssim \norm{(1-\Pi_h^0)f}{H^{-1}(\Omega)} 
    \lesssim h \norm{(1-\Pi_h^0)f}{}
  \end{align*}
  where the last estimate follows from duality arguments.
  For the last two terms in~\eqref{apriori:lsq1:proof1} note that $\uu_h$ is the FOSLS approximation with right-hand side $\Pi_h^0 f$. 
  Thus, we employ Proposition~\ref{apriori:lsq1} with $f$ replaced by $\Pi_h^0 f$, and $u,\ssigma$ replaced by $\widetilde u$, $\nabla\widetilde u$.
  This yields
  \begin{align*}
    \norm{\widetilde u-u_h}{H^1(\Omega)} + \norm{\nabla \widetilde u-\ssigma_h}{} \leq \norm{\widetilde \uu-\uu_h}{W}
    &\lesssim h^{s_\Omega}\norm{\widetilde u}{H^{1+s_\Omega}(\Omega)} + \norm{(1-\Pi_h^0)\Pi_h^0f}{} 
    \\ 
    &\lesssim h^{s_\Omega} \norm{\Pi_h^0 f}{H^{-1+s_\Omega}(\Omega)} 
  \end{align*}
  where we have used the regularity estimate~\eqref{poisson:regshift} and $(1-\Pi_h^0)\Pi_h^0f = 0$.
  The proof is concluded with $\norm{\Pi_h^0f}{H^{-1+s_\Omega}(\Omega)}\lesssim \norm{\Pi_h^0 f}{} \lesssim \norm{f}{}$.
\end{proof}

\begin{remark}
  Proposition~\ref{apriori:lsq2} is stated in a similar form in~\cite[Theorem~4.1]{CaiKu2010} provided that $u\in H^2(\Omega)$ and thus restricted to convex domains where $s_\Omega=1$ is ensured. 
\end{remark}

\section{Regularized MINRES methods}\label{sec:regularization}
The idea of this section is to introduce and analyze regularized versions of the MINRES FEM presented in Section~\ref{sec:minres} 
that allow the use of $H^{-1}(\Omega)$ loads for the Poisson problem
\begin{align}\label{poisson:Hm1}
  -\Delta u &= f, \quad u|_\Gamma = 0.
\end{align}
To that end we define regularization operators in Section~\ref{sec:regOperator}.
Details for a regularized DPG method are found in Section~\ref{sec:regularization:dpg}.
Finally, a regularized FOSLS approach is analyzed in Section~\ref{sec:regularization:ls}.

\subsection{Regularization operator}\label{sec:regOperator}
In this section we present details on the construction of the regularization operator $Q_h^\star$. 
For $f\in H^{-1}(\Omega)\setminus L^2(\Omega)$ it is clear that the functional $F(\vv)$ as considered in Section~\ref{sec:poisson} is not well defined, even if we restrict it to only $\vv\in V_h$.
By replacing the load $f$ with $Q_h^\star f$ we define the regularized load functionals
\begin{align}\label{regFunctional}
  F_h(\vv) := \ip{Q_h^\star f}v \quad\forall \vv=(v,\ttau) \in V.
\end{align}
We consider two variants of $Q_h^\star$ denoted by $P_h'$ and $Q_h$. To this end we follow the presentation given in~\cite[Section~2.4]{MultilevelNorms21}.
Define the quasi-interpolation operator $J_h\colon L^2(\Omega) \to \PP^1(\TT)\cap H_0^1(\Omega)$ by
\begin{align}\label{def:Jh}
  J_h v = \sum_{z\in \mathcal{V}_0} \ip{v}{\psi_z} \eta_z,
\end{align}
where $\eta_z$ denotes the hat-function associated with the vertex $z$, $\norm{\eta_z}{\infty}=1$, and $\psi_z\in\PP^1(\TT)$ with $\supp(\psi_z) = \overline\Patch(z)$ and
\begin{align}\label{def:psiz}
  \psi_z|_{\Patch(z)} = \frac{1}{|\Patch(z)|}\big( (d+1)(d+2)\eta_z-(d+1) \big).
\end{align}
The functions $\psi_z$ are biorthogonal to the hat-functions $\eta_z$ in the sense that $\ip{\psi_z}{\eta_{z'}} = \delta_{z,z'}$ where $\delta_{z,z'}$ denotes the Kronecker-delta. 
We note that $J_h$ is a variant of the Scott--Zhang operator~\cite{SZ_90}, see, e.g.~\cite[Section~2.4]{MultilevelNorms21} and references therein for more details.
Particularly, $J_h$ is a projection, bounded in $L^2(\Omega)$ as well as $H_0^1(\Omega)$, and satisfies (local) approximation propiertes, i.e.,
\begin{align}\label{eq:propJh}
  \norm{(1-J_h)v}{} \lesssim \norm{h_\TT\nabla(1-J_h)v}{} \lesssim \norm{h_\TT\nabla v}{} \quad\text{and}\quad \norm{(1-J_h)w}{H^1(\Omega)} \lesssim \norm{h_\TT D^2 w}{}
\end{align}
for $v\in H_0^1(\Omega)$, $w\in H^2(\Omega)\cap H_0^1(\Omega)$.
Now consider the bubble functions $\eta_{b,T} = \gamma_T \prod_{z\in \mathcal{V}_T} \eta_z$ where $\gamma_T>0$ is a normalization constant chosen such that $\ip{\eta_{b,T}}1_T=1$. Let $\chi_T$ denote the characteristic function on $T$ and 
\begin{align}\label{def:Bh}
  B_h v := \sum_{T\in\TT} \ip{v}{\chi_T} \eta_{b,T} \quad\forall v\in L^2(\Omega).
\end{align}
We note that $B_h\colon L^2(\Omega)\to L^2(\Omega)$ is bounded which can be verified with a direct calculation and the Cauchy--Schwarz inequality.
Furthermore, define $P_h$ by 
\begin{align}\label{def:Ph}
  P_h v:= J_h v + B_h(1-J_h)v.
\end{align}
This is a Fortin-type operator and is constructed so that $(1-P_h)v$ is $L^2$ orthogonal to piecewise constants.
We stress that by construction $P_h$ is idempotent on $\PP^1(\TT)\cap H_0^1(\Omega)$, it is locally bounded in $L^2$ as well as $H^1$ and satisfies local approximation properties, see, e.g.~\cite[Lemma~6]{MultilevelNorms21}.
Particularly, standard approximation results show that $\norm{(1-P_h)v}{H^1(\Omega)} \lesssim h \norm{v}{H^2(\Omega)}$.
We consider its adjoint $P_h'\colon H^{-1}(\Omega)\to \PP^1(\TT)$, 
\begin{align*}
  P_h'\phi = J_h'\phi + (1-J_h')B_h'\phi
\end{align*}
where
\begin{align*}
  J_h'\phi = \sum_{z\in \mathcal{V}_0} \ip{\phi}{\eta_z}\psi_z, 
  \quad
  B_h'\phi = \sum_{T\in\TT} \ip{\phi}{\eta_{b,T}} \chi_T.
\end{align*}
Associated with $P_h'$ we also consider the projector $Q_h := \Pi_h^0 P_h'$, see~\cite[Theorem~8]{MultilevelNorms21}, and recall the following results from~\cite[Section~2.4]{MultilevelNorms21}.
\begin{proposition}[{\cite[Lemma~7 and Theorem~8]{MultilevelNorms21}}]\label{prop:projHoneDual}
  The operator $Q_h^\star\in\{P_h',Q_h\}$ has the following properties:
  \begin{enumerate}[label=(\alph*)]
    \item Idempotent on piecewise constants:\label{prop:projHoneDual:proj}
        $Q_h^\star \phi = \phi$ for all $\phi \in \PP^0(\TT)$.
    \item Approximation:\label{prop:projHoneDual:app}
      $\norm{(1-Q_h^\star)\phi}{H^{-1}(\Omega)} \lesssim \norm{h_\TT\phi}{}$ for all $\phi\in L^2(\Omega)$.
    \item Boundedness:\label{prop:projHoneDual:bound}
      \begin{align*}
        \norm{Q_h^\star\phi}{T} &\lesssim \norm{\phi}{\Patch(T)} 
        \quad\text{for all } T\in\TT,\,\,\phi \in L^2(\Omega), \text{ and }\\
        \norm{Q_h^\star\phi}{H^{-1}(\Omega)} &\lesssim \norm{\phi}{H^{-1}(\Omega)} \text{ for all }\phi \in H^{-1}(\Omega).
      \end{align*}
  \end{enumerate}
\end{proposition}

In the following two statements we present additional properties of the operators $P_h'$ and $Q_h$.
\begin{lemma}\label{lem:propQhstar}
  Let $Q_h^\star\in\{P_h',Q_h\}$ and $f\in H^{-1+s}(\Omega)$ with $s\in[0,1]$. We have that
  \begin{align*}
    \norm{Q_h^\star f}{H^{-1+s}(\Omega)} &\lesssim \norm{f}{H^{-1+s}(\Omega)} \quad\text{and}\\
    \norm{(1-Q_h^\star)f}{H^{-1}(\Omega)} &\lesssim h^s \min_{f_h\in\PP^0(\TT)} \norm{f-f_h}{H^{-1+s}(\Omega)}\leq h^s \norm{f}{H^{-1+s}(\Omega)}.
  \end{align*}
\end{lemma}
\begin{proof}
  The boundedness follows from interpolation estimates. 
  Similarly, the second assertion follows from interpolation estimates and idempotency of the operators on piecewise constants. 
\end{proof}

\begin{lemma}\label{lem:propPh}
  For $f\in H^{-1}(\Omega)$, 
  \begin{align*}
    \norm{(1-P_h')f}{(H^2(\Omega)\cap H_0^1(\Omega))'} \lesssim h\min_{f_h\in \PP^0(\TT)} \norm{f-f_h}{H^{-1}(\Omega)}.
  \end{align*}
\end{lemma}
\begin{proof}
  By Proposition~\ref{prop:projHoneDual} we have that $(1-P_h')f = (1-P_h')(f-f_h)$ for any $f_h\in \PP^0(\TT)$. Thus, with $X=H^2(\Omega)\cap H_0^1(\Omega)$,
  \begin{align*}
    \ip{(1-P_h')f}v_{X'\times X} &= \ip{(1-P_h')(f-f_h)}v = \ip{f-f_h}{(1-P_h)v} 
    \\ &\leq \norm{f-f_h}{H^{-1}(\Omega)}\norm{\nabla(1-P_h)v}{} \lesssim h\norm{f-f_h}{H^{-1}(\Omega)}\norm{v}{H^2(\Omega)}.
  \end{align*}
  The last estimate follows from the properties of $P_h$ discussed above.
\end{proof}

\subsection{Regularized DPG for Poisson}\label{sec:regularization:dpg}
This section is devoted to a regularized DPG method for problem~\eqref{poisson:Hm1}. 
The main results of this section are Theorem~\ref{thm:dpgPoisson} (convergence rates), Theorem~\ref{thm:dpgPostProcess} (postprocessed solution), and Theorem~\ref{thm:dpgPoisson:aposteriori} (a posteriori estimates). 
We consider the DPG method~\eqref{dpg:poisson} with $F$ replaced by the regularized functional $F_h$~\eqref{regFunctional}: Find $\uu_h\in U_h$ such that
\begin{align}\label{dpg:poisson:regF}
  b(\uu_h,\Theta_h\ww_h) = F_h(\Theta_h\ww_h) \quad\forall \ww_h\in U_h.
\end{align}
\begin{theorem}\label{thm:dpgPoisson}
  Let $s\in[0,1]$, $f\in H^{-1+s}(\Omega)$ and $Q_h^\star \in\{P_h',Q_h\}$.
  Furthermore, let $u\in H_0^1(\Omega)$ denote the solution of~\eqref{poisson:Hm1}, $\ssigma:=\nabla u$, $\widehat u:=\tracegrad{} u$, 
  and $\uu_h\in U_h$ denote the solution of~\eqref{dpg:poisson:regF}. The estimate
  \begin{align*}
    \norm{u-u_h}{} + \norm{\ssigma-\ssigma_h}{} + \norm{\widehat u-\widehat u_h}{1/2,\cS}
    \lesssim h^{\min\{s_\Omega,s\}}\norm{f}{H^{-1+\min\{s_\Omega,s\}}(\Omega)}
  \end{align*}
  holds true.

  If $Q_h^\star = P_h'$ and $\Omega$ is convex, then
  \begin{align*}
    \norm{u-u_h}{} \lesssim h\norm{f}{H^{-1}(\Omega)}.
  \end{align*}
\end{theorem}
\begin{proof}
  We take the unique solution $\widetilde u$ of the auxiliary problem
  \begin{align*}
    -\Delta \widetilde u = Q_h^\star f, \quad \widetilde u|_\Gamma = 0
  \end{align*}
  and set $\widetilde\uu = (\widetilde u, \nabla \widetilde u,\tracegrad{}\widetilde u,\tracediv{}\nabla\widetilde u)$.
  Note that $b(\widetilde\uu,\vv)=F_h(\vv)$ for all $\vv\in V$. By the properties of $Q_h^\star$ (Lemma~\ref{lem:propQhstar}) we have that
  with $t:=\min\{s_\Omega,s\}$
  \begin{align*}
    \norm{u-\widetilde u}{H^1(\Omega)} \eqsim \norm{f-Q_h^\star f}{H^{-1}(\Omega)} \lesssim h^t \norm{f}{H^{-1+t}(\Omega)}.
  \end{align*}
  Applying Proposition~\ref{apriori:poisson}, inverse estimates and the boundedness of $Q_h^\star$ leads to
  \begin{align*}
    \norm{\widetilde\uu-\uu_h}U \lesssim h^{s_\Omega}\norm{Q_h^\star f}{H^{-1+s_\Omega}(\Omega)}+h\norm{Q_h^\star f}{} \lesssim h^{t}\norm{Q_h^\star f}{H^{-1+t}(\Omega)} \lesssim h^t\norm{f}{H^{-1+t}(\Omega)}.
  \end{align*}
  Thus the triangle inequality further proves
  \begin{align*}
    \norm{u-u_h}{} + \norm{\nabla u - \ssigma_h}{} + \norm{\widehat u-\widehat u_h}{1/2,\cS} \lesssim h^t \norm{f}{H^{-1+t}(\Omega)},
  \end{align*}
  which is the first assertion.

  For the remaining estimate we need some duality arguments. First, the triangle inequality yields
  \begin{align*}
    \norm{u-u_h}{} &\leq \norm{u-\widetilde u}{} + \norm{\widetilde u - u_h}{}.
  \end{align*}
  We follow the proof of the supercloseness of the $L^2(\Omega)$ projection to the discrete solution~\cite[Proof of Theorem~3]{SupConv2}
  to obtain that $\norm{\Pi_h^0(\widetilde u-u_h)}{}\lesssim h \norm{\widetilde\uu-\uu_h}U$.
  We use $\norm{\widetilde u}{H^1(\Omega)} \lesssim \norm{P_h' f}{H^{-1}(\Omega)}\lesssim  \norm{f}{H^{-1}(\Omega)}$ and an inverse estimate to see that
  \begin{align}\label{dpg:poisson:regF:proof1}
  \begin{split}
    \norm{\widetilde u - u_h}{}  &\leq \norm{(1-\Pi_h^0)\widetilde u}{} + \norm{\Pi_h^0(\widetilde u-u_h)}{} \lesssim h\norm{\nabla \widetilde u}{} + h \norm{\widetilde\uu-\uu_h}{U} 
    \\ &\lesssim h\norm{P_h' f}{H^{-1}(\Omega)} + h^2 \norm{P_h' f}{}
    \lesssim h \norm{P_h' f}{H^{-1}(\Omega)} \lesssim h\norm{f}{H^{-1}(\Omega)}.
  \end{split}
  \end{align}
  Consider the dual problem
  \begin{align*}
    -\Delta v = u-\widetilde u, \quad v|_\Gamma = 0
  \end{align*}
  and note that 
  \begin{align*}
    \norm{u-\widetilde u}{}^2 = \ip{u-\widetilde u}{-\Delta v} = \ip{(1-P_h')f}{v} = \ip{f}{(1-P_h)v} \lesssim \norm{f}{H^{-1}(\Omega)} \norm{(1-P_h)v}{H^1(\Omega)}.
  \end{align*}
  Finally, $\norm{(1-P_h)v}{H^1(\Omega)}\lesssim h\norm{v}{H^2(\Omega)}$ and $\norm{v}{H^2(\Omega)} \lesssim \norm{u-\widetilde u}{}$ conclude the proof.
\end{proof}

\subsubsection{Local postprocessing}\label{sec:postProc}
We follow~\cite{SupConv1,SupConv2} and define the postprocessed solution $u_h^\star\in \PP^1(\TT)$ of the solution $\uu_h=(u_h,\ssigma_h,\widehat u_h,\widehat \sigma_h)\in U_h$ to the regularized problem~\eqref{dpg:poisson:regF} by
\begin{subequations}\label{dpg:poisson:postproc}
  \begin{align}
    \ip{\nabla_\TT u_h^\star}{\nabla_\TT v_h} &= \ip{\ssigma_h}{\nabla_\TT v_h} \quad\forall v_h\in \PP^1(\TT), \\
    \Pi_h^0 u_h^\star &= u_h.
  \end{align}
\end{subequations}

The next result shows that higher rates for the postprocessed solution are achieved when using the regularization operator $Q_h^\star = P_h'$:
\begin{theorem}\label{thm:dpgPostProcess}
  Consider the situation of Theorem~\ref{thm:dpgPoisson} and let $u_h^\star\in \PP^1(\TT)$ denote the postprocessed solution defined by~\eqref{dpg:poisson:postproc}. 
  Then, if $Q_h^\star = P_h'$ and $\Omega$ is convex, 
  \begin{align*}
    \norm{u-u_h^\star}{} \lesssim h^{1+s}\norm{f}{H^{-1+s}(\Omega)}.
  \end{align*}
\end{theorem}
\begin{proof}
  The proof is similar to the proof of the last assertion in Theorem~\ref{thm:dpgPoisson} and~\cite[Theorem~5]{SupConv2}. 
  Consider the solution $\widetilde u\in H_0^1(\Omega)$ of the auxiliary problem $\Delta \widetilde u = -P_h' f$.
  We note that the triangle inequality, $\norm{u-\widetilde u}{}\lesssim \norm{(1-P_h')f}{(H^2(\Omega)\cap H_0^1(\Omega))'}$, cf. the proof of Theorem~\ref{thm:dpgPoisson}, Lemma~\ref{lem:propQhstar} and Lemma~\ref{lem:propPh} yield
  \begin{align*}
    \norm{u-u_h^\star}{} \leq \norm{u-\widetilde u}{} + \norm{\widetilde u-u_h^\star}{} \lesssim h^{1+s}\norm{f}{H^{-1+s}(\Omega)} + \norm{\widetilde u-u_h^\star}{}.
  \end{align*}
  For the last term we apply~\cite[Theorem~5]{SupConv2}, an inverse estimate and boundedness of $P_h'$ to see that 
  \begin{align*}
  \norm{\widetilde u-u_h^\star}{}\lesssim h^2\norm{P_h'f}{} \lesssim h^{1+s}\norm{f}{H^{-1+s}(\Omega)}.
  \end{align*}
This concludes the proof.
\end{proof}

\subsubsection{A posteriori estimator}\label{sec:dpg:aposteriori}
Minimum residual methods like the DPG method come with built-in error estimators that allow to steer adaptive algorithms.
Let $\Pi_h\colon V\to V_h$ denote a Fortin operator, i.e., an operator with $b(\uu_h,v-\Pi_h v)=0$ for all $v\in V$ and $\norm{\Pi_h}{}\lesssim 1$.
For the Poisson problem considered in this work such an operator is constructed in~\cite{practicalDPG}, see also~\cite{DPGaposteriori}. 
We consider the DPG estimator and oscillation terms
\begin{align*}
  \eta&:= \norm{F_h(\cdot)-b(\uu_h,\cdot)}{V_h'}, \\
  \osc(f)&:= \norm{(1-Q_h^\star)f}{H^{-1}(\Omega)}, \\
  \widetilde\osc(Q_h^\star f) &:= \norm{F_h(1-\Pi_h)(\cdot)}{V'}.
\end{align*}

\begin{theorem}\label{thm:dpgPoisson:aposteriori}
  Let $f\in H^{-1}(\Omega)$ and $\TT$ be a regular mesh. Let $\uu_h\in U_h$ and $u\in H_0^1(\Omega)$ denote the solution of~\eqref{dpg:poisson:cont} and~\eqref{poisson}, respectively.  For $Q_h^\star\in \{P_h',Q_h\}$ we have that
  \begin{align*}
    \norm{u-u_h}{}+ \norm{\nabla u-\ssigma_h}{} + \norm{\tracegrad{}u-\widehat u_h}{1/2,\cS} \lesssim \eta + \osc(f) + \widetilde\osc(Q_h^\star f).
  \end{align*}

  Furthermore,
    \begin{align*}
      \eta \lesssim \norm{u-u_h}{}+ \norm{\nabla u-\ssigma_h}{} + \norm{\tracegrad{}u-\widehat u_h}{1/2,\cS} + \osc(f) + \norm{h_\TT(1-\Pi_h^0)Q_h^\star f}{}.
    \end{align*}%
\end{theorem}
\begin{proof}
  As for the a priori analysis we consider the weak solution $\widetilde u\in H_0^1(\Omega)$ of the regularized problem
  \begin{align*}
    -\Delta \widetilde u = Q_h^\star f
  \end{align*}
  and set $\widetilde\uu = (\widetilde u,\nabla \widetilde u,\tracegrad{}\widetilde u,\tracediv{}\nabla \widetilde u)\in U$.
  By~\cite[Theorem~2.1]{DPGaposteriori} we get
  \begin{align*}
    \norm{\widetilde\uu-\uu_h}{U} \eqsim \eta + \widetilde\osc(Q_h^\star f).
  \end{align*}
  The triangle inequality yields
  \begin{align*}
    \norm{u-u_h}{} + \norm{\nabla u-\ssigma_h}{} + \norm{\tracegrad{}u-\widehat u_h}{1/2,\cS} &\leq \norm{u-\widetilde u}{H^1(\Omega)} + \norm{\widetilde\uu-\uu_h}U 
    \\ &\eqsim \norm{(1-Q_h^\star)f}{H^{-1}(\Omega)} + \eta + \widetilde\osc(Q_h^\star f)
  \end{align*}
  which finishes the proof of the reliability estimate.

  Using that $\eta\lesssim \norm{\widetilde\uu-\uu_h}U$ and the best-approximation property $\norm{\widetilde\uu-\uu_h}U\lesssim \norm{\widetilde\uu-\vv_h}U$ for all $\vv\in U_h$ we get with the triangle inequality that 
  \begin{align*}
    \eta\lesssim \norm{\widetilde\uu-\uu_h}U &\lesssim \norm{\widetilde u-u_h}{} + \norm{\widetilde\ssigma-\ssigma_h}{} + \norm{\tracegrad{}\widetilde u-\widehat u_h}{1/2,\cS}
    + \norm{\tracediv{}\widetilde\ssigma-\tracediv{}\Pi_h^\mathrm{div}\widetilde\ssigma}{-1/2,\cS}
    \\
    &\leq \norm{u-u_h}{} + \norm{\ssigma-\ssigma_h}{} + \norm{\tracegrad{}u-\widehat u_h}{1/2,\cS} 
    \\
    &\qquad +\norm{u-\widetilde u}{} + \norm{\ssigma-\widetilde\ssigma}{} + \norm{\tracegrad{}u-\tracegrad{}\widetilde u}{1/2,\cS}
    + \norm{\tracediv{}\widetilde\ssigma-\tracediv{}\Pi_h^\mathrm{div}\widetilde\ssigma}{-1/2,\cS}
    \\
    &\lesssim \norm{u-u_h}{} + \norm{\ssigma-\ssigma_h}{} + \norm{\tracegrad{}u-\widehat u_h}{1/2,\cS} 
    \\ &\qquad + \osc(f) + \norm{\tracediv{}\widetilde\ssigma-\tracediv{}\Pi_h^\mathrm{div}\widetilde\ssigma}{-1/2,\cS},
  \end{align*}
  where we have used that $\norm{u-\widetilde u}{H^1(\Omega)} + \norm{\ssigma-\widetilde\ssigma}{} \eqsim \osc(f)$. 
  It remains to estimate the term $\norm{\tracediv{}\widetilde\ssigma-\tracediv{}\Pi_h^\mathrm{div}\widetilde\ssigma}{-1/2,\cS}$. To do so we use the following identity from~\cite[Lemma~2.2]{breakSpace}:
  \begin{align*}
    \norm{\tracediv{}\ttau}{-1/2,\cS} = \sup_{0\neq v\in H^1(\TT)} \frac{\dual{\tracediv{}\ttau}v_{\cS}}{\norm{v}{H^1(\TT)}}.
  \end{align*}
  Using that $\Pi_h^\mathrm{div}$ is a projection,~\eqref{eq:propertiesPihDiv}, $\div\widetilde\ssigma=-Q_h^\star f$ and approximation properties we arrive at 
  \begin{align*}
    \dual{\tracediv{}\widetilde\ssigma-\tracediv{}\Pi_h^\mathrm{div}\widetilde\ssigma}v_{\cS} 
    &= \ip{(1-\Pi_h^0)\div\widetilde\ssigma}v + \ip{(1-\Pi_h^\mathrm{div})\widetilde\ssigma}{\nabla_\TT v}
    \\
    &= \ip{(1-\Pi_h^0)(-Q_h^\star f)}{(1-\Pi_h^0)v} + \ip{(1-\Pi_h^\mathrm{div})(\widetilde\ssigma-\ssigma_h)}{\nabla_\TT v}
    \\
    &\lesssim \Big( \norm{h_\TT(1-\Pi_h^0)Q_h^\star f}{} + \norm{\widetilde\ssigma - \ssigma_h}{} \Big)\norm{\nabla_\TT v}{}.
  \end{align*}
  We conclude with the triangle inequality and $\norm{\ssigma-\widetilde\ssigma}{}\lesssim \osc(f)$ that
  \begin{align*}
    \norm{\tracediv{}\widetilde\ssigma-\tracediv{}\Pi_h^\mathrm{div}\widetilde\ssigma}{-1/2,\cS}
    \lesssim \norm{\widetilde\ssigma-\ssigma_h}{} + \widetilde\osc(Q_h^\star f) \lesssim \norm{\ssigma-\ssigma_h}{} + \osc(f) +
    \widetilde\osc(Q_h^\star f)
  \end{align*}
  which finishes the proof.
\end{proof}

\begin{remark}
If the polynomial degree of the discrete test space is increased, i.e., $V_h = \PP^{d+1}(\TT)\times \PP^2(\TT)^d$ for $Q_h^\star = Q_h$, $V_h =\PP^{d+2}(\TT)\times \PP^2(\TT)^d$ for $Q_h^\star = P_h'$, then $\widetilde\osc(Q_h^\star f)=0$. This can be easily seen from the properties of the Fortin operator, cf.~\cite[Eq.(3.4a)]{DPGaposteriori}.
Note that $\norm{h_\TT(1-\Pi_h^0)Q_h^\star f}{}$ vanishes if $Q_h^\star = Q_h$. 
If $Q_h^\star = P_h'$, then by an inverse estimate and properties of $P_h'$, 
\begin{align*}
  \norm{h_\TT(1-\Pi_h^0)P_h'f}{} &= \norm{h_\TT(1-\Pi_h^0)P_h'(1-Q_h)f}{} \leq \norm{h_\TT P_h'(1-Q_h)f}{}
  \\ 
  &\lesssim \norm{P_h'(1-Q_h)f}{H^{-1}(\Omega)} \lesssim \norm{(1-Q_h)f}{H^{-1}(\Omega)}.
\end{align*}
\end{remark}

\subsection{Regularized FOSLS for Poisson}\label{sec:regularization:ls}
This section is devoted to the study of a regularized FOSLS for the Poisson problem that allows to use $H^{-1}(\Omega)$ loads. 
The main results are Theorem~\ref{thm:LSregularized} (convergence rates) and Theorem~\ref{thm:lsq:aposteriori} (a posteriori estimates).

We replace $f\in L^2(\Omega)$ in~\eqref{lsq:minimization:discrete} by $Q_h^\star f$, i.e., we consider the minimization problem
\begin{align}\label{lsq:reg}
  \uu_h = \argmin_{\vv_h=(v_h,\ttau_h)\in W_h}\big( \norm{\nabla v_h-\ttau_h}{}^2 + \norm{\div\ttau_h+Q_h^\star f}{}^2 \big).
\end{align}
Let us note that the Euler--Lagrange equations read:
\begin{align}\label{lsq:eulerLag:reg}
  \ip{\div\ssigma_h}{\div\ttau_h} + \ip{\nabla u_h-\ssigma_h}{\nabla v_h-\ttau_h} = \ip{-Q_h^\star f}{\div\ttau_h} 
  \quad\forall \vv_h=(v_h,\ttau_h)\in W_h.
\end{align}
Recalling that $W_h = (H_0^1(\Omega)\cap \PP^1(\TT))\times \RT^0(\TT)$ and $\div(\RT^0(\TT))=\PP^0(\TT)$ we find that
\begin{align*}
  \ip{-Q_h^\star f}{\div\ttau_h} = \ip{-Q_h f}{\div\ttau_h} \quad\forall \ttau_h\in \RT^0(\TT)
\end{align*}
for $Q_h^\star \in\{P_h',Q_h\}$. 
Therefore, the right-hand side in~\eqref{lsq:eulerLag:reg} is the same for either operator, $P_h'$ or $Q_h$, 
and we can restrict the analysis to $Q_h^\star = Q_h$ for the remainder of this section.

\begin{theorem}\label{thm:LSregularized}
  For $s\in[0,1]$ and $f\in H^{-1+s}(\Omega)$, let $u\in H_0^1(\Omega)$ denote the solution of the Poisson problem with right-hand side $f$ and let $\uu_h$ denote the solution of~\eqref{lsq:reg} with $Q_h^\star=Q_h$.
The estimate
  \begin{align*}
    \norm{u-u_h}{H^1(\Omega)} + \norm{\nabla u-\ssigma_h}{} \lesssim h^{\min\{s_\Omega,s\}}\norm{f}{H^{-1+\min\{s_\Omega,s\}}(\Omega)}
  \end{align*}
  holds true.
\end{theorem}
\begin{proof}
  In most parts the proof is the same as for Proposition~\ref{apriori:lsq2}.
  Consider
  \begin{align*}
    -\Delta \widetilde u = Q_h f, \quad \widetilde u|_\Gamma = 0.
  \end{align*}
  Note that $\uu_h\in W_h$ is the least-squares approximation to $\widetilde\uu=(\widetilde u,\nabla \widetilde u)$.
  By the properties of the operator $Q_h$ (Lemma~\ref{lem:propQhstar}) we have with $t:=\min\{s_\Omega,s\}$ that
  \begin{align*}
    \norm{u-\widetilde u}{H^1(\Omega)} \eqsim \norm{(1-Q_h)f}{H^{-1}(\Omega)} \lesssim h^t \norm{f}{H^{-1+t}(\Omega)}.
  \end{align*}
  Proposition~\ref{apriori:lsq2} (replacing $f$ with $Q_h f$ and $\uu=(u,\ssigma)$ with $\widetilde\uu$) shows that
  \begin{align*}
    \norm{\widetilde u-u_h}{H^1(\Omega)} + \norm{\nabla \widetilde u-\ssigma_h}{} &\lesssim h^{s_\Omega}\norm{Q_h f}{H^{-1+s_\Omega}(\Omega)} + h \norm{(1-\Pi_h^0)Q_h f}{}.
  \end{align*}
  The last term vanishes since $Q_hf\in \PP^0(\TT)$. 
  An inverse estimate and the boundedness of $Q_h$ yield
  \begin{align*}
    h^{s_\Omega}\norm{Q_h f}{H^{-1+s_\Omega}(\Omega)} &\lesssim h^t\norm{Q_h f}{H^{-1+t}(\Omega)}
    \lesssim h^t\norm{f}{H^{-1+t}(\Omega)}.
  \end{align*}
  The proof is concluded using the triangle inequality.
\end{proof}

\subsubsection{A posteriori estimate}\label{sec:lsq:aposteriori}
We can also use the least-squares functional to measure, up to an oscillation term, errors. To that end define the estimator and oscillation term
\begin{align*}
  \eta &:= \left(\norm{\nabla u_h-\ssigma_h}{}^2 + \norm{\div\ssigma_h+Q_h f}{}^2\right)^{1/2}, \\
  \osc(f) &:= \norm{(1-Q_h)f}{H^{-1}(\Omega)}.
\end{align*}

\begin{theorem}\label{thm:lsq:aposteriori}
  Let $f\in H^{-1}(\Omega)$ and $\TT$ be a regular mesh. If $u\in H_0^1(\Omega)$ denotes the solution of the Poisson problem and $\uu_h=(u_h,\ssigma_h)\in W_h$ the solution of~\eqref{lsq:reg} with $Q_h^\star = Q_h$, then, we have that
  \begin{align*}
    \eta \lesssim \norm{u-u_h}{H^1(\Omega)}+ \norm{\nabla u-\ssigma_h}{} \lesssim \eta + \osc(f).
  \end{align*}
\end{theorem}
\begin{proof}
  For the proof of the upper bound consider the regularized problem
  \begin{align*}
    -\Delta\widetilde u = Q_h f, \quad \widetilde u|_\Gamma = 0
  \end{align*}
  and set $\widetilde \uu=(\widetilde u,\nabla\widetilde u)\in U$. Since $\uu_h$ is the FOSLS approximation of $\widetilde \uu$ we have that
  \begin{align*}
    \norm{\widetilde \uu-\uu_h}W \eqsim \eta. 
  \end{align*}
  Thus, the triangle inequality yields the estimate
  \begin{align*}
    \norm{u-u_h}{H^1(\Omega)} + \norm{\nabla u-\ssigma_h}{} \lesssim \norm{(1-Q_h)f}{H^{-1}(\Omega)} + \eta.
  \end{align*}

  For the lower bound we use the operator $\Pi_h^\mathrm{div}\colon \Hdivset\Omega\to \RT^0(\TT)$ (see Section~\ref{sec:approx} and~\eqref{eq:propertiesPihDiv}). We show that $\norm{\div\ssigma_h+Q_h f}{}\lesssim \norm{\nabla u_h-\ssigma_h}{}$.
  To this end we consider the unique weak solution $\vv=(v,\ttau)\in W$ to the problem
  \begin{align*}
    \div\ttau &= -(\div\ssigma_h+Q_h f), \\
    \nabla v-\ttau &= 0, \\
    v|_\Gamma &= 0.
  \end{align*}
  With $\vv_h:=(v_h,\ttau_h):=(J_hv,\Pi_h^\mathrm{div}\ttau)\in W_h$ (see~\eqref{def:Jh} for the definition of the operator $J_h$) and Galerkin orthogonality we infer that
  \begin{align*}
  \norm{\div\ssigma_h+Q_h f}{}^2 &= \ip{-Q_h f-\div\ssigma_h}{\div\ttau} + \ip{-\nabla u_h+\ssigma_h}{\nabla v - \ttau} \\ 
    &=-\ip{\div\ssigma_h+Q_h f}{\div(\ttau-\ttau_h)} - \ip{\nabla u_h-\ssigma_h}{\nabla(v-v_h)-(\ttau-\ttau_h)}.
  \end{align*}
  The commutativity property $\div(\ttau-\ttau_h)=(1-\Pi_h^0)\div\ttau$ and $\div\ttau\in \PP^0(\TT)$ show that the first term on the right-hand side vanishes. 
  Boundedness of $J_h$, $\Pi_h^\mathrm{div}$ (see~\eqref{eq:propertiesPihDiv}) and stability of the Poisson problem prove
  \begin{align*}
    \norm{v-J_h v}{H^1(\Omega)} + \norm{\ttau-\Pi_h^\mathrm{div}\ttau}{} 
    &\lesssim \norm{v}{H^1(\Omega)} + \norm{\ttau}{}
    \lesssim  \norm{\div\ssigma_h+Q_h f}{H^{-1}(\Omega)}.
  \end{align*}
  Putting the latter observations together implies with $\norm{\cdot}{H^{-1}(\Omega)}\lesssim \norm{\cdot}{}$ that
  \begin{align*}
    \norm{\div\ssigma_h+Q_h f}{}^2 &= -\ip{\div\ssigma_h+Q_h f}{\div(\ttau-\ttau_h)} - \ip{\nabla u_h-\ssigma_h}{\nabla(v-v_h)-(\ttau-\ttau_h)} 
    \\&\lesssim \norm{\nabla u_h-\ssigma_h}{} (\norm{v-J_h v}{H^1(\Omega)} + \norm{\ttau-\Pi_h^\mathrm{div}\ttau}{}) \lesssim \norm{\nabla u_h-\ssigma_h}{}\norm{\div\ssigma_h+Q_h f}{}.
  \end{align*}
  With the triangle inequality we conclude that 
  \begin{align*}
    \eta \lesssim \norm{\nabla u_h-\ssigma_h}{} \lesssim \norm{\nabla(u-u_h)}{} + \norm{\nabla u-\ssigma_h}{}
  \end{align*}
  which finishes the proof.
\end{proof}
The assertions of Theorem~\ref{thm:lsq:aposteriori} are known for $f\in L^2(\Omega)$ and $Q_h^\star=\Pi_h^0$, see~\cite[Theorem~2]{CC2020}. An equivalence similar to~\cite[Eq.(3)]{CC2020} holds, as stated in the following result:

\begin{corollary}
  Under the assumptions of Theorem~\ref{thm:lsq:aposteriori} suppose additionally that $f\in L^2(\Omega)$. 
  The equivalence
  \begin{align*}
    &\norm{\nabla u_h-\ssigma_h}{} + \norm{\div\ssigma_h+Q_h f}{} + \norm{h_\TT(1-\Pi_h^0)f}{} 
    \\&\qquad\eqsim \norm{\nabla u-\ssigma_h}{} + \norm{u-u_h}{H^1(\Omega)} + \norm{h_\TT(1-\Pi_h^0)f}{}
  \end{align*}
  holds true.
\end{corollary}
\begin{proof}
  The equivalence follows from Theorem~\ref{thm:lsq:aposteriori}, the estimate 
  \begin{align*}
    \norm{(1-Q_h)f}{H^{-1}(\Omega)} = \norm{(1-Q_h)(1-\Pi_h^0)f}{H^{-1}(\Omega)} \lesssim \norm{h_\TT(1-\Pi_h^0)f}{}
  \end{align*}
  which is due to the projection property of $Q_h$, and Lemma~\ref{lem:propQhstar}.
\end{proof}

\section{On the optimality of $L^2$ error estimates in the FOSLS}\label{sec:L2errLS}
In this section we focus on $L^2(\Omega)$ error estimates in the primal variable of the solutions $(u_h,\ssigma_h)$ of~\eqref{lsq:minimization:discrete} resp.~\eqref{lsq:reg} given that $f\in L^2(\Omega)$.
That is, we consider the approximations
\begin{align}\label{lsq:minimization:Qhstar}
  (u_h,\ssigma_h) = \argmin_{(v_h,\ttau_h)\in W_h} \norm{\div\ttau_h + Q_h^\star f}{}^2 + \norm{\nabla v_h-\ttau_h}{}^2
\end{align}
with $Q_h^\star = \Pi_h^0$ (standard FOSLS) or $Q_h^\star = Q_h$ (regularized FOSLS). 

For the solution component $u_h$ with $Q_h^\star =\Pi_h^0$, 
error estimates in $L^2(\Omega)$ have been studied, e.g., in~\cite{Ku2011} and references therein.
For a study of optimal $L^2(\Omega)$ convergence rates on smooth domains we refer to the recent article~\cite{LSQL2normSmooth}.
There, the authors prove optimal convergence rates for the standard FOSLS with higher-order discretization spaces, whereas the case of the lowest-order space $W_h$ is excluded, see~\cite[Theorem~4.13 and Remark~4.14]{LSQL2normSmooth}.

The following result is similar to~\cite[Theorem~4.5]{Ku2011}, but we do not require sufficiently small mesh-sizes and can handle both cases $Q_h^\star = \Pi_h^0$ and $Q_h^\star = Q_h$ simultaneously. 
For simplicity we restrict the presentation to convex domains.

\begin{theorem}\label{thm:ls:L2}
  Suppose that $\Omega$ is convex. For $f\in L^2(\Omega)$, let $u\in H_0^1(\Omega)$ denote the solution of~\eqref{poisson} and $(u_h,\ssigma_h)\in W_h$ the solution of~\eqref{lsq:minimization:Qhstar} with either $Q_h^\star = \Pi_h^0$ or $Q_h^\star = Q_h$. We have that
  \begin{align*}
    \norm{u-u_h}{} \lesssim h^2\norm{f}{} + \norm{(1-Q_h^\star)f}{(H^2(\Omega)\cap H_0^1(\Omega))'}.
  \end{align*}
\end{theorem}
\begin{proof}
  Let $\widetilde u \in H_0^1(\Omega)$ be the solution of 
  \begin{align*}
    -\Delta \widetilde u = Q_h^\star f.
  \end{align*}
  By the triangle inequality and regularity estimates we get that 
  \begin{align*}
    \norm{u-u_h}{} \leq \norm{u-\widetilde u}{} + \norm{\widetilde u-u_h}{} \lesssim \norm{(1-Q_h^\star)f}{(H^2(\Omega)\cap H_0^1(\Omega))'} + \norm{\widetilde u-u_h}{}.
  \end{align*}
  It remains to estimate $\norm{\widetilde u-u_h}{}$ which we will do by employing a duality argument, see, e.g.,~\cite{CaiKu2006}.
  First, define $w\in H_0^1(\Omega)$ as the solution of $-\Delta w = \widetilde u-u_h$ and $\vv:=(v,\ttau)\in W$ as the solution of
  \begin{align*}
    \div\ttau &= -w, \\
    \nabla v -\ttau &= \nabla w.
  \end{align*}
  Second, with $\widetilde\ssigma :=\nabla\widetilde u$, integration by parts shows that
  \begin{align*}
    \norm{\widetilde u-u_h}{}^2 &= \ip{\widetilde u-u_h}{-\Delta w} = \ip{\nabla(\widetilde u-u_h)}{\nabla w} 
    \\
    &=  \ip{\nabla(\widetilde u-u_h)-(\widetilde\ssigma-\ssigma_h)}{\nabla w} + \ip{\div(\widetilde\ssigma-\ssigma_h)}{-w} \\
    &= \ip{\nabla(\widetilde u-u_h)-(\widetilde\ssigma-\ssigma_h)}{\nabla v-\ttau} + \ip{\div(\widetilde\ssigma-\ssigma_h)}{\div\ttau}.
  \end{align*}
  Finally, we argue as in the proof of Theorem~\ref{thm:lsq:aposteriori} employing the operator $\Pi_h^\mathrm{div}$.
  By using Galerkin orthogonality, choosing $\vv_h=(v_h,\ttau_h)=(J_h v,\Pi_h^\mathrm{div}\ttau)$, and $\div(\widetilde\ssigma-\ssigma_h)\in\PP^0(\TT)$  we see that 
  \begin{align*}
    &\ip{\nabla(\widetilde u-u_h)-(\widetilde\ssigma-\ssigma_h)}{\nabla v-\ttau} + \ip{\div(\widetilde\ssigma-\ssigma_h)}{\div\ttau}
    \\
    &\qquad= \ip{\nabla(\widetilde u-u_h)-(\widetilde\ssigma-\ssigma_h)}{\nabla (v-v_h)-(\ttau-\ttau_h)} + \ip{\div(\widetilde\ssigma-\ssigma_h)}{\div(\ttau-\ttau_h)}
    \\
    &\qquad \lesssim (\norm{\widetilde u-u_h}{H^1(\Omega)} + \norm{\widetilde\ssigma-\ssigma_h}{})(\norm{\nabla(v-v_h)}{}+\norm{\ttau-\ttau_h}{}).
  \end{align*}
  Recalling that $\uu_h$ is the FOSLS approximation of $\widetilde\uu=(\widetilde u,\widetilde\ssigma)$ we may employ Proposition~\ref{apriori:lsq2} with $f$ replaced by $Q_h^\star f$ to infer that
  \begin{align*}
    \norm{\widetilde u-u_h}{H^1(\Omega)} + \norm{\widetilde\ssigma-\ssigma_h}{} \lesssim h\norm{Q_h^\star f}{}\lesssim h\norm{f}{}.
  \end{align*}
  The properties of the operators $J_h$ and $\Pi_h^\mathrm{div}$ (see~\eqref{eq:propertiesPihDiv}) together with elliptic regularity show that $\norm{\nabla(v-v_h)}{}+\norm{\ttau-\ttau_h}{}\lesssim h\norm{\widetilde u-u_h}{}$. This finishes the proof. 
\end{proof}

\begin{remark}\label{rem:extendLSL2}
  Theorem~\ref{thm:ls:L2} with $Q_h^\star =Q_h$ can be extended to $f\in H^{-1+s}(\Omega)$, $s\in[0,1]$. The same argumentation yields
  \begin{align*}
    \norm{u-u_h}{} \lesssim h^{1+s}\norm{f}{H^{-1+s}(\Omega)} + \norm{(1-Q_h)f}{(H^2(\Omega)\cap H_0^1(\Omega))'}.
  \end{align*}
\end{remark}

\begin{remark}
  By duality arguments and the projection property of $Q_h^\star$ one sees that
  \begin{align*}
    \norm{(1-Q_h^\star)f}{(H^2(\Omega)\cap H_0^1(\Omega))'} \lesssim \norm{(1-Q_h^\star)f}{H^{-1}(\Omega)} \lesssim h \norm{(1-\Pi_h^0)f}{}.
  \end{align*}
  Thus, additional regularity $f\in H^1(\TT)$ proves $\norm{u-u_h}{}=\OO(h^2)$.
  We note that this has also been observed in~\cite[Remark~4.2]{Ku2011} for $Q_h^\star = \Pi_h^0$.

  We note that, usually, $\norm{(1-\Pi_h^0)f}{(H^2(\Omega)\cap H_0^1(\Omega))'}$ will not converge at $\OO(h^2)$ without further regularity of $f$. 
  Under some conditions on the mesh, superconvergence of $\norm{(1-Q_h)f}{(H^2(\Omega)\cap H_0^1(\Omega))'}$ can be proven (see Section~\ref{sec:meshcondition} below).

  In Section~\ref{sec:ex:L2notoptimal} we present a numerical example in 2D for which $\norm{u-u_h}{}=\OO(h^2)$ if $Q_h^\star = Q_h$ but $\norm{u-u_h}{}\neq \OO(h^2)$ if $Q_h^\star=\Pi_h^0$.
\end{remark}

\subsection{Optimal $L^2$ estimate under mesh condition}\label{sec:meshcondition}
For $T\in\TT$ let $s_T\in \R^d$ denote its center of mass, i.e., $s_T = \frac1{d+1}\sum_{z\in \VV_T} z$.
For each interior node $z\in\VV_0$ we define the center of mass of the patch $\patch(z)$ by
\begin{align*}
  s_z := \frac{1}{|\Patch(z)|}\sum_{T\in\patch(z)} |T| s_T.
\end{align*}

For the analysis we use the Cl\'ement interpolation operator~\cite{Clement75}, $J_h^\mathrm{Cl\'e} v := \sum_{z\in\VV_0} V_z \eta_z$  with zero-th order moments
\begin{align*}
  V_z := \frac1{|\Patch(z)|} \int_{\Patch(z)} v(x)\,\di x,  \quad z\in \VV_0.
\end{align*}
Furthermore, recall the definitions of $J_h$, $\psi_z$, $B_h$ and $P_h$, cf.~\eqref{def:Jh}--\eqref{def:Ph}.

The following observation is crucial: 
\begin{lemma}
For $v\in L^2(\Omega)$ we have that
\begin{align}\label{eq:clementIdentity}
  P_h \Pi_h^0 v = J_h^\mathrm{Cl} v + B_h(1-J_h^\mathrm{Cl})v.
\end{align}
\end{lemma}
\begin{proof}
Note that $\ip{\psi_z}{1}_T = \frac{|T|}{|\Patch(z)|}$ yielding
\begin{align*}
  \ip{\psi_z}{\Pi_h^0 v} &= \sum_{T\in\patch(z)} \Pi_h^0v|_T \ip{\psi_z}{1}_T = \sum_{T\in\patch(z)} \frac{|T|}{|\Patch(z)|} \Pi_h^0v|_T 
  = \sum_{T\in\patch(z)} \frac{1}{|\Patch(z)|} \ip{v}1_T  
  = V_z.
\end{align*}
This proves that $J_h\Pi_h^0 v = J_h^\mathrm{Cl}v$. Note that $\ip{\Pi_h^0v}1_T = \ip{v}1_T$ implies $B_h\Pi_h^0 v = B_h v$.
Putting all the identities together and using that $P_h = J_h + B_h(1-J_h)$ we obtain~\eqref{eq:clementIdentity}.
\end{proof}

The following superconvergence result holds on special meshes:
\begin{lemma}\label{lem:clement}
  Suppose that $s_z = z$ for all $z\in \VV_0$. 
  For $v\in H^2(\Omega)\cap H_0^1(\Omega)$ we have that
  \begin{align*}
    \norm{v-P_h\Pi_h^0 v}{}\lesssim h^2 \norm{v}{H^2(\Omega)}.
  \end{align*}
\end{lemma}
\begin{proof}
  Identity~\eqref{eq:clementIdentity} and boundedness of $B_h\colon L^2(\Omega)\to L^2(\Omega)$, see~\eqref{def:Bh}, prove that 
  \begin{align*}
    \norm{v-P_h\Pi_h^0 v}{} = \norm{v-J_h^\mathrm{Cl} v - B_h(1-J_h^\mathrm{Cl})v}{} \lesssim  \norm{(1-J_h^\mathrm{Cl})v}{}
    \leq \norm{v-J_h v}{} + \norm{J_h v- J_h^\mathrm{Cl}v}{}.
  \end{align*}
  For the first term on the right-hand side we use~\eqref{eq:propJh}, i.e., $\norm{v-J_h v}{} \lesssim h^2\norm{v}{H^2(\Omega)}$. 
  For the second term we note that by the $L^2(\Omega)$ stability of the basis functions we have that
  \begin{align}\label{eq:proofClement1}
    \norm{(J_h-J_h^\mathrm{Cl})v}{}^2 \eqsim \sum_{z\in\VV_0} |(J_h-J_h^\mathrm{Cl})v(z)|^2\norm{\eta_z}{}^2 \eqsim 
    \sum_{z\in\VV_0}|\Patch(z)|\,|(J_h-J_h^\mathrm{Cl})v(z)|^2 
  \end{align}
  where $(J_h-J_h^\mathrm{Cl})v(z) = \ip{v}{\psi_z}-\frac1{\Patch(z)}\ip{v}1_{\Patch(z)}$.
  Let $q$ be a polynomial of degree $\leq1$. 
  The properties of $\psi_z$ prove that
  \begin{align*}
    \ip{q}{\psi_z} = q(z).
  \end{align*}
  Furthermore, we stress that $\ip{q}1_T = |T|q(s_T)$ yielding
  \begin{align*}
    \frac{1}{|\Patch(z)|} \ip{q}{1}_{\Patch(z)} = \frac{1}{|\Patch(z)|} \sum_{T\in\patch(z)}|T| q(s_T) = q(s_z).
  \end{align*}
  Thus, under the assumption $s_z = z$ for all $z\in \VV_0$ the equality $J_hq(z) = J_h^\mathrm{Cl}q(z)$ holds and, consequently, with $\Pi_z^1 v$ the $L^2$ projection on polynomials of degree $\leq 1$, 
  \begin{align*}
    |(J_h-J_h^\mathrm{Cl})v(z)| &= |(J_h-J_h^\mathrm{Cl})(v-\Pi_z^1v)(z)| =|\ip{v-\Pi_z^1 v}{\psi_z}-\frac1{\Patch(z)}\ip{v-\Pi_z^1 v}1_{\Patch(z)}| \\
    &\lesssim |\Patch(z)|^{-1/2} \norm{v-\Pi_z^1v}{\Patch(z)} \lesssim |\Patch(z)|^{-1/2}h^2\norm{v}{H^2(\Patch(z))}.
  \end{align*}
  Combining the last estimate with~\eqref{eq:proofClement1}, the equivalence $\sum_{z\in\VV_0}\norm{v}{H^2(\Omega(z))}^2\eqsim \norm{v}{H^2(\Omega)}^2$ finishes the proof.
\end{proof}

Lemma~\ref{lem:propPh} holds true for the operator $Q_h$ under the mesh condition:
\begin{lemma}\label{lem:propQh}
  Suppose that $s_z=z$ for all $z\in\VV_0$. 
  For $f\in H^{-1}(\Omega)$, 
  \begin{align*}
    \norm{(1-Q_h)f}{(H^2(\Omega)\cap H_0^1(\Omega))'} \lesssim h\min_{f_h\in \PP^0(\TT)} \norm{f-f_h}{H^{-1}(\Omega)}.
  \end{align*}
\end{lemma}
\begin{proof}
  The proof is similar to the proof of Lemma~\ref{lem:propPh}: Let $X = H^2(\Omega)\cap H_0^1(\Omega)$.
  Recall that $Q_h = \Pi_h^0 P_h'$. For $f_h\in\PP^0(\TT)$,
  \begin{align*}
    \ip{(1-Q_h)f}{v}_{X'\times X} &= \ip{(1-Q_h)(f-f_h)}v = \ip{f-f_h}{v-P_h\Pi_h^0 v} 
    \\ 
    &\lesssim \norm{f-f_h}{H^{-1}(\Omega)} \norm{\nabla(v-P_h\Pi_h^0 v)}{}.
  \end{align*}
  Finally, an inverse estimate and Lemma~\ref{lem:clement} yield
  \begin{align*}
    \norm{\nabla(v-P_h\Pi_h^0 v)}{}& \lesssim \norm{\nabla (v-P_hv)}{} + h^{-1}\norm{P_h v -P_h\Pi_h^0 v}{}
    \\
    &\lesssim \norm{\nabla(1-P_h)v}{} + h^{-1}\norm{(1-P_h\Pi_h^0) v}{} \lesssim h\norm{v}{H^2(\Omega)},
  \end{align*}
  which concludes the proof.
\end{proof}

With the results from this section one has optimal convergence rates for the $L^2(\Omega)$ error:
\begin{corollary}\label{cor:ls:L2}
  Under the situation of Theorem~\ref{thm:ls:L2} with $Q_h^\star = Q_h$, suppose additionally that $s_z=z$ for all $z\in\VV_0$.
  We have that
  \begin{align*}
    \norm{u-u_h}{} \lesssim  h^2 \norm{f}{}.
  \end{align*}
\end{corollary}
\begin{proof}
  This follows from Theorem~\ref{thm:ls:L2}, Lemma~\ref{lem:propQh} and $\min_{f_h\in \PP^0(\TT)} \norm{f-f_h}{H^{-1}(\Omega)}\lesssim h\norm{f}{}$.
\end{proof}

\begin{remark}
  The mesh condition in this section is used to prove Corollary~\ref{cor:ls:L2}. 
  In numerical experiments (not presented in this work) we found that the assertion of Corollary~\ref{cor:ls:L2} is observed even though the mesh condition $s_z=z$ for all $z\in \VV_0$ is not met. 
\end{remark}

\section{Point loads}\label{sec:pointloads}
Throughout this section we consider a fixed $x_0 \in \Omega$ and let $\delta_{x_0}$ denote the corresponding Dirac delta distribution.
We are interested in approximating the solution of
\begin{align}\label{delta:poisson}
  -\Delta u = \delta_{x_0}, \quad u|_\Gamma = 0.
\end{align}
Clearly, $\delta_{x_0}\notin H^{-1}(\Omega)$. 
While in related works, cf.~\cite{Scott75}, the evaluation of discrete test functions at $x_0$ is well defined, this may not be the case here due to discontinuities of test functions across elements, e.g., the DPG method uses a subspace of $H^1(\TT)$. 
In~\cite{HoustonWihler12}, which deals with discontinuous Galerkin methods, it is assumed that $x_0$ lies in an element interior.
In order to avoid such an assumption we consider a regularized delta distribution. 
The point evaluations of the hat- resp. bubble-functions, $\eta_z$ resp. $\eta_{b,T}$, are well defined. 
Consequently, we can allow delta distributions as arguments for the operators $P_h',Q_h$ from Section~\ref{sec:regOperator}.
The remainder of this section shows how to extend the analysis of regularized MINRES FEM from Section~\ref{sec:regularization} to point loads. 

We need the following technical lemma:
\begin{lemma}\label{deltaNormScaling} 
  Let $Q_h^\star\in\{P_h',Q_h\}$. We have that
  \begin{align*}
    \norm{Q_h^\star \delta_{x_0}}{} \eqsim h^{-d/2}.
  \end{align*}
  Moreover, if $d=2$ and $h$ is sufficiently small, then
  \begin{align*}
    \norm{Q_h^\star \delta_{x_0}}{H^{-1}(\Omega)} \lesssim |\log h|^{1/2}.
  \end{align*}
  If $d=3$ then
  \begin{align*}
    \norm{Q_h^\star \delta_{x_0}}{H^{-1}(\Omega)} \eqsim h^{-1/2}.
  \end{align*}
  
  Furthermore,
  \begin{align*}
    \norm{\delta_{x_0}-P_h'\delta_{x_0}}{(H^2(\Omega)\cap H_0^1(\Omega))'}\lesssim h^{2-d/2},
  \end{align*}
  and, if $s_z=z$ for all $z\in \VV_0$, then
  \begin{align*}
    \norm{\delta_{x_0}-Q_h\delta_{x_0}}{(H^2(\Omega)\cap H_0^1(\Omega))'}\lesssim h^{2-d/2}.
  \end{align*}
\end{lemma}
\begin{proof}
  The proof of $\norm{Q_h^\star \delta_{x_0}}{} \eqsim h^{-d/2}$
  follows with the same techniques as, e.g., in~\cite[Section~3.1]{HoustonWihler12} (for $d=2$, the case $d=3$ is similar), see also~\cite[Theorem~1]{Scott75}. 

  The estimate in the negative norm follows from the local support of $Q_h^\star\delta_{x_0}$ and scaling properties of basis functions in $H^{-1}(\Omega)$ which can be found in~\cite[Theorem~4.8]{amt99}.

  Finally, we have for $v\in X:=H^2(\Omega)\cap H_0^1(\Omega)$ that
  \begin{align*}
    \ip{(1-P_h')\delta_{x_0}}v_{X'\times X}  = (1-P_h)v(x_0).
  \end{align*}
  Choose a $T\in\TT$ with $x_0\in \overline T$. A scaling argument and the approximation properties of $P_h$ show
  \begin{align*}
    |(1-P_h)v(x_0)| \lesssim \frac{1}{|T|^{1/2}}(\norm{(1-P_h)v}T + h^2\norm{D^2 v}T)
    \lesssim \frac{1}{|T|^{1/2}} h^2 \norm{v}{H^2(\Omega)}. 
  \end{align*}
  Note that $|T|^{-1/2} h^2 \eqsim h^{2-d/2}$.
  For $P_h'$ replaced by $Q_h$ we argue similarly by using the results from Section~\ref{sec:meshcondition}.
\end{proof}

\subsection{DPG with point loads}
We consider the DPG problem (see Section~\ref{sec:regularization:dpg} for details):
Find $\uu_h\in U_h$ such that
\begin{align}\label{dpg:poisson:delta}
  b(\uu_h,\Theta_h\ww_h) = F_h(\Theta_h\ww_h) \quad\forall \ww_h\in U_h.
\end{align}
Here, $F_h(\vv) := \ip{P_h'\delta_{x_0}}v$ for all $\vv=(v,\ttau)\in V$. 

\begin{theorem}\label{thm:dpg:delta}
  Suppose that $\Omega$ is convex. Let $u\in L^2(\Omega)$ denote the solution of~\eqref{delta:poisson} and let
  $\uu_h=(u_h,\ssigma_h,\widehat u_h,\widehat\sigma_h)\in U_h$ denote the solution of~\eqref{dpg:poisson:delta}. 
  We have (for $h$ sufficiently small)
  \begin{align*}
    \norm{u-u_h}{} \lesssim \begin{cases}
      |\log h|^{1/2}\, h & d=2, \\
      h^{1/2} & d=3,
    \end{cases}
  \end{align*}
  where the involved constant only depends on $\Omega$, the distance of $x_0$ to the boundary $\Gamma$, and the shape-regularity constant of $\TT$.
\end{theorem}
\begin{proof}
  Consider the auxiliary problem: Find $\widetilde u\in H_0^1(\Omega)$ with
  \begin{align*}
    -\Delta \widetilde u = P_h'\delta_{x_0}.
  \end{align*}

  From the proof of~\cite[Theorem~1]{Scott75} together with Lemma~\ref{deltaNormScaling} we infer that
  \begin{align*}
    \norm{u-\widetilde u}{} \leq C(x_0) \norm{(1-P_h')\delta_{x_0}}{(H^2(\Omega)\cap H_0^1(\Omega))'} \lesssim h^{2-d/2},
  \end{align*}
  where the constant $C(x_0)$ depends on the distance of $x_0$ to the boundary $\Gamma$.

  Following the arguments as in the proof of Theorem~\ref{thm:dpgPoisson} we deduce that (see~\eqref{dpg:poisson:regF:proof1})
  \begin{align*}
    \norm{\widetilde u-u_h}{} \lesssim h \norm{P_h' \delta_{x_0}}{H^{-1}(\Omega)} + h^2\norm{P_h'\delta_{x_0}}{}.
  \end{align*}
  Combination of all estimates and Lemma~\ref{deltaNormScaling} conclude the proof.
\end{proof}

\subsection{FOSLS with point loads}
We consider the problem
\begin{align}\label{lsq:minimization:delta}
  (u_h,\ssigma_h) = \argmin_{\vv_h=(v_h,\ttau_h)\in W_h}\big( \norm{\nabla v_h-\ttau_h}{}^2 + \norm{\div\ttau_h+Q_h\delta_{x_0}}{}^2 \big).
\end{align}

\begin{theorem}\label{thm:lsq:delta}
  Suppose that $\Omega$ is convex. Let $u\in L^2(\Omega)$ denote the solution of~\eqref{delta:poisson} and let
  $\uu_h=(u_h,\ssigma_h)\in W_h$ denote the solution of~\eqref{lsq:minimization:delta}. 
  If $s_z=z$ for all $z\in\VV_0$, then
  \begin{align*}
    \norm{u-u_h}{} \lesssim h^{2-d/2},
  \end{align*}
  where the involved constant only depends on $\Omega$, the distance of $x_0$ to the boundary $\Gamma$, and the shape-regularity constant of $\TT$.
\end{theorem}
\begin{proof}
  The proof follows the lines of the proof of Theorem~\ref{thm:ls:L2} in combination with the results from Lemma~\ref{deltaNormScaling} and the idea from the proof of Theorem~\ref{thm:dpg:delta} with obvious modifications.
\end{proof}

\section{Numerical examples}\label{sec:ex}

\subsection{Example DPG for Poisson}\label{sec:ex:DPGPoisson}
We consider the Poisson problem with manufactured solution
\begin{align}\label{ex:solutionPoisson}
  u(x,y) = |x-y|^{3/4} \sin(\pi x)\sin(\pi y)  \quad (x,y)\in \Omega:=(0,1)^2.
\end{align}
One verifies that $u\in H^{1+1/4-\eps}(\Omega)$ and $f:=-\Delta u \in H^{-1+1/4-\varepsilon}(\Omega)$ for all $\varepsilon>0$.
We consider the DPG method with regularization operator $Q_h^\star = P_h'$. Results are shown in the left plot of Figure~\ref{fig:DPGPoissonUniform}. 
We visualize $\eta$ (error estimator), $\norm{u-u_h}{}$, $\norm{\ssigma-\ssigma_h}{}$, and $\norm{u-u_h^\star}{}$ where $u_h^\star$ is the post-processed solution (Section~\ref{sec:postProc}).
The dotted black lines correspond to $\OO(h^{1/4})$, $\OO(h)$, $\OO(h^{1+1/4})$. The expected optimal rates are (omitting $\eps$)
\begin{align*}
  \norm{u-u_h}{} = \OO(h), \quad \norm{\ssigma-\ssigma_h}{} = \OO(h^{1/4}), \quad \norm{u-u_h^\star}{} = \OO(h^{1+1/4}),
\end{align*}
which are indeed observed in the experiment. They perfectly fit the theory (Theorem~\ref{thm:dpgPoisson} and Theorem~\ref{thm:dpgPostProcess}).
\begin{figure}
  \begin{center}
    \begin{tikzpicture}
\begin{loglogaxis}[
    title={DPG},
    width=0.49\textwidth,
cycle list/Dark2-6,
cycle multiindex* list={
mark list*\nextlist
Dark2-6\nextlist},
every axis plot/.append style={ultra thick},
xlabel={degrees of freedom},
grid=major,
legend entries={\small $\eta$,\small $\|u-u_h\|$,\small $\|u-u^\star_h\|$,\small $\|\ssigma-\ssigma_h\|$},
legend pos=south west,
]
\addplot table [x=dofDPG,y=estDPG] {data/DPGPhprime.dat};
\addplot table [x=dofDPG,y=errU] {data/DPGPhprime.dat};
\addplot table [x=dofDPG,y=errUtilde] {data/DPGPhprime.dat};
\addplot table [x=dofDPG,y=errSigma] {data/DPGPhprime.dat};
\addplot [black,dotted,mark=none] table [x=dofDPG,y expr={0.35*sqrt(\thisrowno{1})^(-1-1/4)}] {data/DPGPhprime.dat};
\addplot [black,dotted,mark=none] table [x=dofDPG,y expr={0.6*sqrt(\thisrowno{1})^(-1)}] {data/DPGPhprime.dat};
\addplot [black,dotted,mark=none] table [x=dofDPG,y expr={0.8*sqrt(\thisrowno{1})^(-1/4)}] {data/DPGPhprime.dat};
\end{loglogaxis}
\end{tikzpicture}
\begin{tikzpicture}
\begin{loglogaxis}[
    title={FOSLS},
    width=0.49\textwidth,
cycle list/Dark2-6,
cycle multiindex* list={
mark list*\nextlist
Dark2-6\nextlist},
every axis plot/.append style={ultra thick},
xlabel={degrees of freedom},
grid=major,
legend entries={\small $\eta$,\small $\|u-u_h\|$,\small $\|\nabla(u-u_h)\|$,\small $\|\ssigma-\ssigma_h\|$},
legend pos=south west,
]
\addplot table [x=dofLSQ,y=estLSQ] {data/LSQQh.dat};
\addplot table [x=dofLSQ,y=errUL2] {data/LSQQh.dat};
\addplot table [x=dofLSQ,y=errUH1] {data/LSQQh.dat};
\addplot table [x=dofLSQ,y=errSigmaL2] {data/LSQQh.dat};
\addplot [black,dotted,mark=none] table [x=dofLSQ,y expr={0.35*sqrt(\thisrowno{1})^(-1-1/4)}] {data/LSQQh.dat};
\addplot [black,dotted,mark=none] table [x=dofLSQ,y expr={0.6*sqrt(\thisrowno{1})^(-1)}] {data/LSQQh.dat};
\addplot [black,dotted,mark=none] table [x=dofLSQ,y expr={0.8*sqrt(\thisrowno{1})^(-1/4)}] {data/LSQQh.dat};
\end{loglogaxis}
\end{tikzpicture}
  \end{center}
  \caption{Errors and estimators for the DPG (left) and FOSLS (right) methods for the problem from Sections~\ref{sec:ex:DPGPoisson},~\ref{sec:ex:LSPoisson}.
  The dotted black lines correspond to $\OO(h^{1/4})$, $\OO(h^1)$ and $\OO(h^{1+1/4})$.}\label{fig:DPGPoissonUniform}
\end{figure}
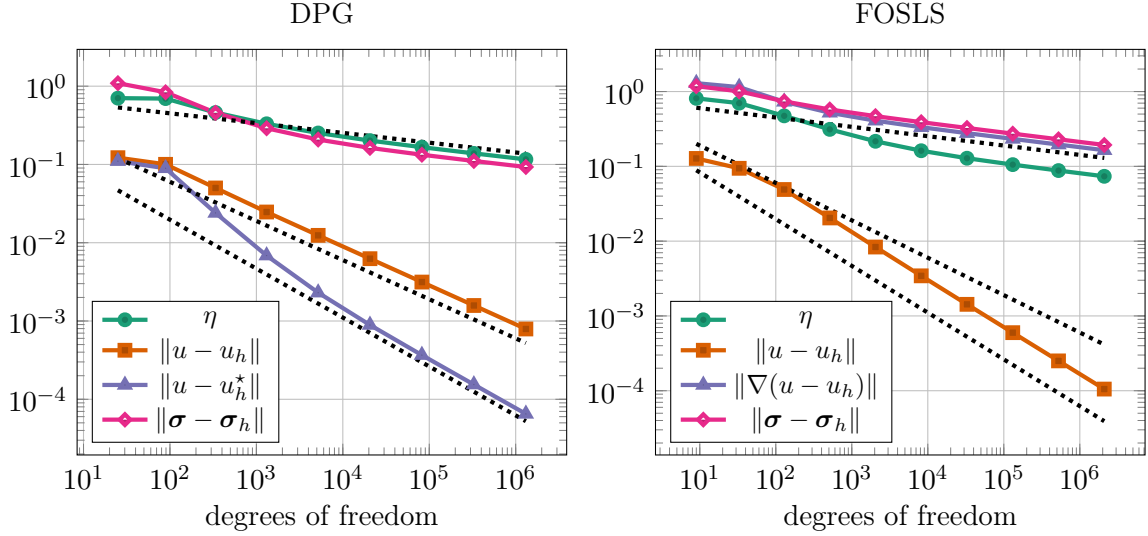

\subsection{Example FOSLS for Poisson}\label{sec:ex:LSPoisson}
We consider the regularized FOSLS~\eqref{lsq:minimization:Qhstar} and the same setup as in Section~\ref{sec:ex:DPGPoisson}.
Results are presented in the right plot of Figure~\ref{fig:DPGPoissonUniform} where we plot
$\norm{\nabla(u-u_h)}{}$, $\norm{\ssigma-\ssigma_h}{}$, $\eta$ (error estimator), $\norm{u-u_h}{}$.
From Theorem~\ref{thm:LSregularized} we expect that (omitting $\eps$)
\begin{align*}
  \norm{\nabla(u-u_h)}{} + \norm{\ssigma-\ssigma_h}{} = \OO(h^{1/4}),
\end{align*}
which is also observed.
Moreover, we find that $\norm{u-u_h}{} = \OO(h^{1+1/4})$ which is the optimal rate for the $L^2(\Omega)$ error.
In this experiment the meshes satisfy the condition $s_z = z$ for all $z\in \VV_0$ so that the optimal rate for $\norm{u-u_h}{}$ is covered by our theory (see Remark~\ref{rem:extendLSL2}).

\subsection{Example FOSLS with and without optimal $L^2(\Omega)$ rate}\label{sec:ex:L2notoptimal}
We consider the domain $\Omega = (-1,1)^2$ and the manufactured solution $u(x,y) = v(x)w(y)$ where 
\begin{align*}
  v(x) = x \,|x|^{1/2+1/128}(1-x^2), \qquad w(y) = 1-y^2.
\end{align*}
Note that $u\in H^2(\Omega)$, particularly,
\begin{align*}
  f(x,y) &= -\Delta u(x,y) = \frac{\sign(x)\big(144129 \, x^2-12545 \big)}{65536 \, |x|^{63/128}}w(y)+2v(x).
\end{align*}
One verifies that $f\in L^2(\Omega)$ but $f\notin H^t(\Omega)$ for $t\geq 1/128$.

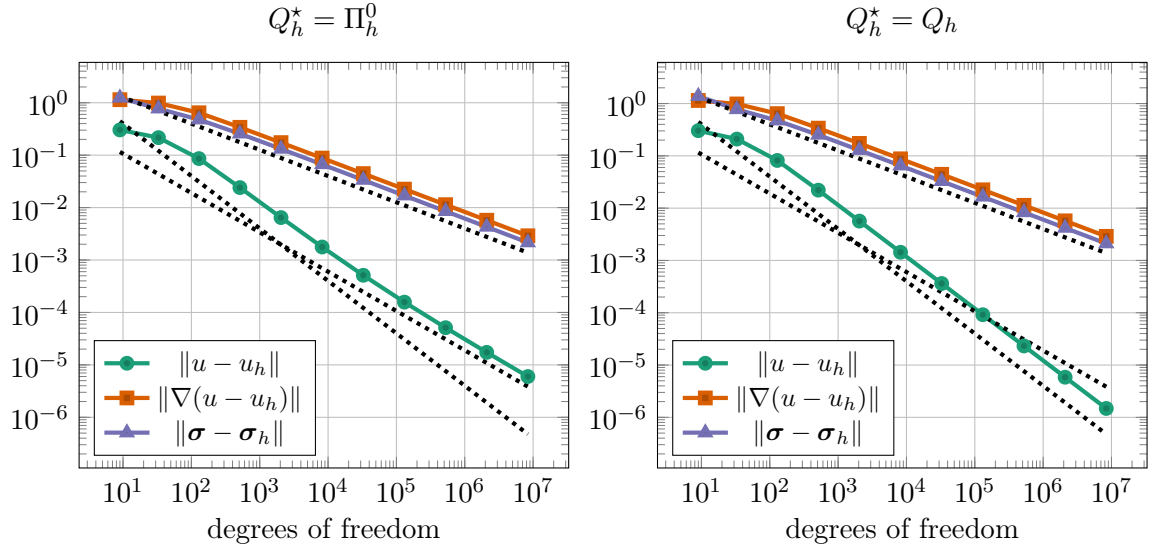
\begin{figure}
  \begin{center}
    \begin{tikzpicture}
\begin{loglogaxis}[
    title={$Q_h^\star = \Pi_h^0$},
    width=0.49\textwidth,
cycle list/Dark2-6,
cycle multiindex* list={
mark list*\nextlist
Dark2-6\nextlist},
every axis plot/.append style={ultra thick},
xlabel={degrees of freedom},
grid=major,
legend entries={\small $\|u-u_h\|$,\small $\|\nabla(u-u_h)\|$,\small $\|\ssigma-\ssigma_h\|$},
legend pos=south west,
]
\addplot table [x=dofLSQ,y=errUL2] {data/LSQerrL2Pih.dat};
\addplot table [x=dofLSQ,y=errUH1] {data/LSQerrL2Pih.dat};
\addplot table [x=dofLSQ,y=errSigmaL2] {data/LSQerrL2Pih.dat};
\addplot [black,dotted,mark=none] table [x=dofLSQ,y expr={4*sqrt(\thisrowno{1})^(-1)}] {data/LSQerrL2Pih.dat};
\addplot [black,dotted,mark=none] table [x=dofLSQ,y expr={0.6*sqrt(\thisrowno{1})^(-3/2)}] {data/LSQerrL2Pih.dat};
\addplot [black,dotted,mark=none] table [x=dofLSQ,y expr={4*sqrt(\thisrowno{1})^(-2)}] {data/LSQerrL2Pih.dat};
\end{loglogaxis}
\end{tikzpicture}
\begin{tikzpicture}
\begin{loglogaxis}[
    title={$Q_h^\star = Q_h$},
    width=0.49\textwidth,
cycle list/Dark2-6,
cycle multiindex* list={
mark list*\nextlist
Dark2-6\nextlist},
every axis plot/.append style={ultra thick},
xlabel={degrees of freedom},
grid=major,
legend entries={\small $\|u-u_h\|$,\small $\|\nabla(u-u_h)\|$,\small $\|\ssigma-\ssigma_h\|$},
legend pos=south west,
]
\addplot table [x=dofLSQ,y=errUL2] {data/LSQerrL2Qh.dat};
\addplot table [x=dofLSQ,y=errUH1] {data/LSQerrL2Qh.dat};
\addplot table [x=dofLSQ,y=errSigmaL2] {data/LSQerrL2Qh.dat};
\addplot [black,dotted,mark=none] table [x=dofLSQ,y expr={4*sqrt(\thisrowno{1})^(-1)}] {data/LSQerrL2Qh.dat};
\addplot [black,dotted,mark=none] table [x=dofLSQ,y expr={0.6*sqrt(\thisrowno{1})^(-3/2)}] {data/LSQerrL2Qh.dat};
\addplot [black,dotted,mark=none] table [x=dofLSQ,y expr={4*sqrt(\thisrowno{1})^(-2)}] {data/LSQerrL2Qh.dat};
\end{loglogaxis}
\end{tikzpicture}
  \end{center}
  \caption{Errors for the FOSLS~\eqref{lsq:minimization:Qhstar} with $Q_h^\star =\Pi_h^0$ (left) and $Q_h^\star=Q_h$ (right) for the problem from Section~\ref{sec:ex:L2notoptimal}. The black dotted lines indicate $\OO(h)$, $\OO(h^{3/2})$ and $\OO(h^2)$.}\label{fig:LSL2notoptimal}
\end{figure}

Consider the solution $\uu_h=(u_h,\ssigma_h)\in W_h$ of the FOSLS~\eqref{lsq:minimization:Qhstar} with $Q_h^\star =\Pi_h^0$ and $Q_h^\star = Q_h$.
In Figure~\ref{fig:LSL2notoptimal} we plot the errors $\norm{\nabla(u-u_h)}{}$, $\norm{\ssigma-\ssigma_h}{}$, $\norm{u-u_h}{}$.
The three dotted lines indicate $\OO(h)$, $\OO(h^{3/2})$, $\OO(h^2)$. 
The left plot shows the results for $Q_h^\star = \Pi_h^0$ and the right plot shows the results for $Q_h^\star = Q_h$. 
For $Q_h^\star = \Pi_h^0$ we note that the error $\norm{u-u_h}{}$, although pre-asymptotically converges at the optimal rate, i.e., $\OO(h^2)$, seems to converge at $\OO(h^{3/2})$. 
For $Q_h^\star = Q_h$ we find optimal rates also for $\norm{u-u_h}{}$. We note that the meshes satisfy the condition $s_z=z$ for all $z\in \VV_0$, so that the optimal rate for $\norm{u-u_h}{}$ is covered by the theory (see Corollary~\ref{cor:ls:L2}).

We conclude that even though $f\in L^2(\Omega)$, the regularization approach ($Q_h^\star = Q_h$) delivers more accurate solutions compared to the standard method ($Q_h^\star = \Pi_h^0$). 

\subsection{Example DPG for Poisson with point load}\label{sec:ex:DPGpoint}
Let $\Omega = (-1,1)^2$, $x_0=(0,0)$ and $u\in L^2(\Omega)$ be the solution of~\eqref{delta:poisson}. 
Let $\uu_h=(u_h,\ssigma_h,\widehat u_h,\widehat\sigma_h)\in U_h$ denote the solution of~\eqref{dpg:poisson:delta}. 
Figure~\ref{fig:pointLoads} (left) shows the error $\norm{u-u_h}{}$ which numerically confirms the results from Theorem~\ref{thm:dpg:delta}. The black dotted line corresponds to $\OO(h)$. 
We also plot the error of the postprocessed solution $\norm{u-u_h^\star}{}$ (Section~\ref{sec:postProc}) which seems to give slightly better approximations (we have not analyzed convergence of the postprocessed solution for point loads in this work). 

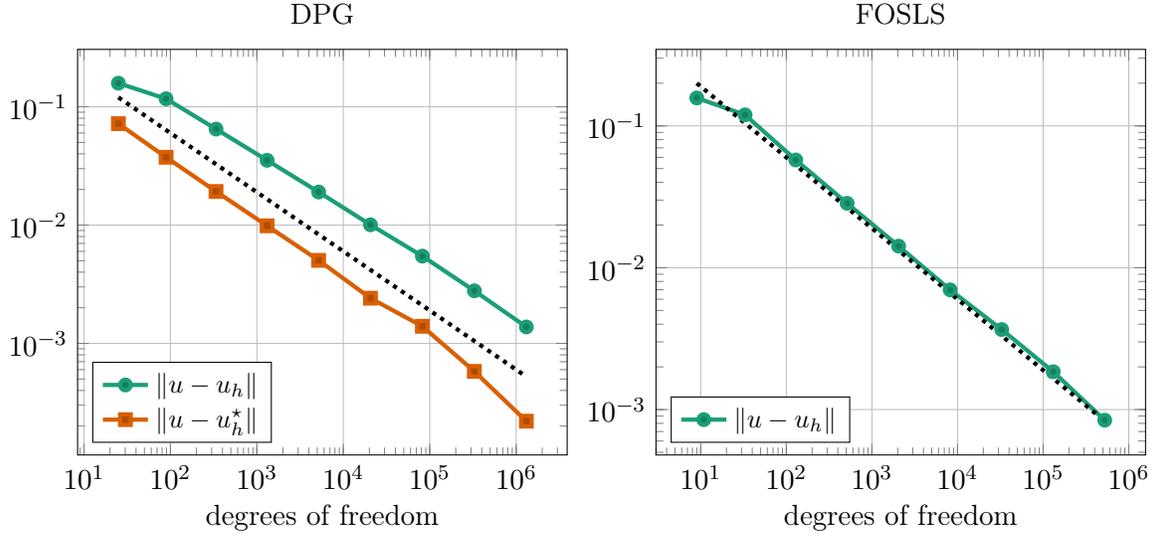
\begin{figure}
  \begin{center}
    \begin{tikzpicture}
\begin{loglogaxis}[
    title={DPG},
    width=0.49\textwidth,
cycle list/Dark2-6,
cycle multiindex* list={
mark list*\nextlist
Dark2-6\nextlist},
every axis plot/.append style={ultra thick},
xlabel={degrees of freedom},
grid=major,
legend entries={\small $\|u-u_h\|$,\small $\|u-u_h^\star\|$},
legend pos=south west,
]
\addplot table [x=dofDPG,y=errU] {data/DPGpointLoad.dat};
\addplot table [x=dofDPG,y=errUtilde] {data/DPGpointLoad.dat};
\addplot [black,dotted,mark=none] table [x=dofDPG,y expr={0.6*sqrt(\thisrowno{1})^(-1)}] {data/DPGpointLoad.dat};
\end{loglogaxis}
\end{tikzpicture}
\begin{tikzpicture}
\begin{loglogaxis}[
    title={FOSLS},
    width=0.49\textwidth,
cycle list/Dark2-6,
cycle multiindex* list={
mark list*\nextlist
Dark2-6\nextlist},
every axis plot/.append style={ultra thick},
xlabel={degrees of freedom},
grid=major,
legend entries={\small $\|u-u_h\|$},
legend pos=south west,
]
\addplot table [x=dofLSQ,y=errUL2] {data/LSQpointLoad.dat};
\addplot [black,dotted,mark=none] table [x=dofLSQ,y expr={0.6*sqrt(\thisrowno{1})^(-1)}] {data/LSQpointLoad.dat};
\end{loglogaxis}
\end{tikzpicture}
  \end{center}
  \caption{Errors for the DPG~\eqref{dpg:poisson:delta} (left) and FOSLS~\eqref{lsq:minimization:delta} (right) method with point source (Sections~\ref{sec:ex:DPGpoint} and~\ref{sec:ex:LSpoint}). 
  The black dotted line indicates $\OO(h)$.}\label{fig:pointLoads}
\end{figure}

\subsection{Example FOSLS for Poisson with point load}\label{sec:ex:LSpoint}
We consider the same problem setup as in Section~\ref{sec:ex:DPGpoint} and let $\uu_h=(u_h,\ssigma_h)$ denote the solution of~\eqref{lsq:minimization:delta}. 
We note that the meshes satisfy the condition $s_z=z$ for all $z\in\VV_0$, so that by Theorem~\ref{thm:lsq:delta} we expect $\norm{u-u_h}{}=\OO(h)$. This is indeed observed in Figure~\ref{fig:pointLoads}.

\bibliographystyle{abbrv}
\bibliography{literature}

\appendix
\section{Extension}\label{sec:extension}
In this section we study a possible extension of the regularized MINRES methods.
For the sake of brevity we only consider the extension of the regularized FOSLS to the problem
\begin{align}\label{poisson:general}
\begin{split}
  -\div A\nabla u + K u &= f \quad\text{in } \Omega, \\
  u|_{\Gamma_D} &= 0, \\ 
  \normal\cdot A\nabla u |_{\Gamma_N}&= 0,
\end{split}
\end{align}
where $\Gamma_D$, $\Gamma_N$ denotes a disjoint decomposition of the boundary $\Gamma$ with surface measure $|\Gamma_D|>0$. 
Here, $A\in L^\infty(\Omega;\R_\mathrm{sym}^{d\times d})$ is uniformly positive definite and bounded, i.e., there exist $c>0$, $C>0$ with
\begin{align*}
 c y^\top y\leq y^\top A(x)y \leq C y^\top y \quad\text{for all }y\in\R^d, \text{ and } x\in\Omega \text{ a.e.}
\end{align*}
Different choices for the bounded operator $K\colon H^1(\Omega)\to L^2(\Omega)$ are possible, e.g.,
\begin{align*}
  Ku = \aalpha \cdot\nabla u + \beta u, \quad \aalpha \in L^\infty(\Omega)^d, \, \beta\in L^\infty(\Omega),
\end{align*}
see~\cite[Eq.(2.6)]{CaiLazarovManteuffelMcCormickPart1}.
The choice $A=\operatorname{Id}$, $\aalpha=0$, $\beta <0$ corresponds to the Helmholtz problem. 
With 
\begin{align*}
W:= H_D^1(\Omega) \times \HdivsetN\Omega := \set{v\in H^1(\Omega)}{v|_{\Gamma_D}=0}\times\set{\ttau\in\Hdivset\Omega}{\ttau\cdot\normal|_{\Gamma_N}=0}
\end{align*}
we consider the first-order reformulation: Find $\uu=(u,\ssigma)\in W$ such that
\begin{align*}
  -\div\ssigma + Ku &= f, \\
  A\nabla u -\ssigma &=0
\end{align*}
and with $W_h := (\PP^1(\TT)\cap \RT^0(\TT))\cap W$ the minimization problem
\begin{align}\label{fosls:general}
  (u_h,\ssigma_h) = \argmin_{\vv_h=(v_h,\ttau_h)\in W_h} \norm{\div\ttau_h-Kv_h+f}{}^2 + \norm{A\nabla v_h-\ttau_h}{}^2.
\end{align}
We assume that problem~\eqref{poisson:general} induces an isomorfism, i.e., for each $f\in H_D^{-1}(\Omega):=(H_D^1(\Omega))'$ there exists a unique $u\in H_D^1(\Omega)$ with
\begin{align*}
  \norm{u}{H^1(\Omega)} \lesssim \norm{f}{H_D^{-1}(\Omega)}.
\end{align*}
In~\cite[Theorem~3.1]{CaiLazarovManteuffelMcCormickPart1} well-posedness of the FOSLS~\eqref{fosls:general} was shown for $f\in L^2(\Omega)$. 
To conclude convergence rates we additionally assume that there exists $s_\Omega\in[0,1]$ such that
\begin{align}\label{poisson:regshift:general}
  \norm{u}{H^{1+s}(\Omega)} + \norm{A\nabla u}{H^{s}(\Omega)} \lesssim \norm{f}{H_D^{-1+s}(\Omega)} \quad\text{for } s\in[0,s_\Omega].
\end{align}
Here, $H_D^{-t}(\Omega) = (H_D^t(\Omega))'$ and $H_D^t(\Omega)$ is defined by interpolation of $L^2(\Omega)$ and $H_D^1(\Omega)$. 

To define a regularized FOSLS we redefine the operator $J_h$ from Section~\ref{sec:regOperator}. With $\VV_D = \VV\setminus\Gamma_N$ we set
\begin{align*}
  J_h v:= \sum_{z\in \VV_D} \ip{\psi_z}v \eta_z, \quad\text{and}\quad P_h := J_h + B_h(1-J_h).
\end{align*}
Recall from Section~\ref{sec:regOperator} that $Q_h=\Pi_h^0P_h'$. It is straightforward to verify that $J_h$, $P_h$, $P_h'$, $Q_h$ satisfy properties corresponding to the ones in Section~\ref{sec:regOperator} (Proposition~\ref{prop:projHoneDual}, Lemma~\ref{lem:propQhstar}, Lemma~\ref{lem:propPh}).
Moreover, we use the notation $\Pi_h^\mathrm{div}$ for the projector $\HdivsetN\Omega\to \RT^0(\TT)\cap \HdivsetN\Omega$ from~\cite[Section~3]{egsv2019} which has the properties~\eqref{eq:propertiesPihDiv} with $\ttau\in\Hdivset\Omega$ replaced by $\ttau\in\HdivsetN\Omega$. 

The regularized FOSLS reads: 
\begin{align}\label{lsq:reg:general}
  \uu_h=(u_h,\ssigma_h) = \argmin_{\vv_h=(v_h,\ttau_h)\in W_h} \norm{\div\ttau_h-Kv_h+Q_hf}{}^2 + \norm{A\nabla v_h-\ttau_h}{}^2.
\end{align}

We show how to extend Theorem~\ref{thm:LSregularized} to the problem described in this section.
\begin{theorem}\label{thm:LSregularized:general}
  For $s\in[0,1]$ and $f\in H_D^{-1+s}(\Omega)$, let $u\in H_D^1(\Omega)$ denote the solution of~\eqref{poisson:general} with right-hand side $f$ and let $\uu_h$ denote the solution of~\eqref{lsq:reg:general}.
  The estimate
  \begin{align*}
    \norm{u-u_h}{H^1(\Omega)} + \norm{A\nabla u-\ssigma_h}{} \lesssim h^{\min\{s_\Omega,s\}}\norm{f}{H_D^{-1+\min\{s_\Omega,s\}}(\Omega)}
  \end{align*}
  holds true.
\end{theorem}
\begin{proof}
  Let $\widetilde u\in H_D^1(\Omega)$ denote the solution of~\eqref{poisson:general} with right-hand side $Q_hf$. With stability of the problem and the properties of $Q_h$
  we conclude
  \begin{align*}
    \norm{u-\widetilde u}{H^1(\Omega)} + \norm{\ssigma-\widetilde\ssigma}{} \lesssim \norm{u-\widetilde u}{H^1(\Omega)}
    \lesssim \norm{(1-Q_h)f}{H^{-1}_D(\Omega)} \lesssim h^t\norm{f}{H^{-1+t}_D(\Omega)},
  \end{align*}
  where $t:=\min\{s_\Omega,s\}$, and $\ssigma=A\nabla u$, $\widetilde\ssigma = A\nabla\widetilde u$.

  Choose $\vv_h = (J_h \widetilde u,\Pi_h^\mathrm{div}\widetilde\ssigma)$. The quasi-best approximation of the FOSLS (see~\cite{CaiLazarovManteuffelMcCormickPart1}) implies that 
  \begin{align*}
    \norm{\widetilde\uu-\uu_h}W \lesssim \norm{\widetilde\uu-\vv_h}{W} \leq \norm{(1-J_h)\widetilde u}{H^1(\Omega)} + \norm{(1-\Pi_h^\mathrm{div})\widetilde\ssigma}{} +
    \norm{\div(1-\Pi_h^\mathrm{div})\widetilde\ssigma}{}.
  \end{align*}
  The commutativity property of $\Pi_h^\mathrm{div}$ and $Q_h f\in \PP^0(\TT)$ show that $\div(1-\Pi_h^\mathrm{div})\widetilde\ssigma = 0$.
  Approximation properties together with~\eqref{poisson:regshift:general} and boundedness of $Q_h$ yield 
  \begin{align*}
    \norm{(1-J_h)\widetilde u}{H^1(\Omega)} + \norm{(1-\Pi_h^\mathrm{div})\widetilde\ssigma}{} \lesssim h^t\norm{Q_hf}{H^{-1+t}_D(\Omega)} \lesssim h^t\norm{f}{H^{-1+t}_D(\Omega)}.
  \end{align*}
  The triangle concludes the proof. 
\end{proof}

\end{document}